%% file: main.tex
\documentclass[11pt]{amsart}
\usepackage[utf8]{inputenc}
\usepackage{amssymb,latexsym, parskip}

\pdfoutput = 1

\usepackage{amsmath}
\usepackage{graphicx}
\usepackage{textcomp}

\usepackage{amsthm,amssymb,enumerate,graphicx, tikz}
\usetikzlibrary{decorations.markings}
\usetikzlibrary{decorations.pathmorphing}
\usetikzlibrary{patterns, patterns.meta}
\usepackage{amscd}
\usepackage{setspace}
\usepackage{mathrsfs}
\usepackage{comment}
\usepackage{hyperref}
\usepackage[capitalise]{cleveref}
\usepackage{subcaption}
\usepackage{enumerate}
\usepackage{standalone}

\usepackage[left=1in,right=1in,top=1in,bottom=1in,bindingoffset=0cm]{geometry}

\hypersetup{
    colorlinks=true,
    linkcolor=red,
    citecolor=green,
    filecolor=magenta,
    urlcolor=blue,
    pdfborder={0 0 0}, 
}

\newtheorem{theorem}{Theorem}[section]

\newtheorem*{theorem*}{Theorem}
\newtheorem{proposition}[theorem]{Proposition}
\newtheorem{lemma}[theorem]{Lemma}
\newtheorem{observation}[theorem]{Observation}

\theoremstyle{definition}
\newtheorem{definition}[theorem]{Definition}

\theoremstyle{remark}
\newtheorem{remark}[theorem]{Remark}

\newtheorem{claim}[theorem]{Claim}
\newenvironment{proofc}{\begin{proof}[Proof of Claim]}{\end{proof}}

\newcommand{\G}{\mathcal{G}}

\renewcommand{\l}{\ell}

\newcommand{\C}{\mathcal{C}}

\newcommand{\T}{\mathcal{T}}

\definecolor{green}{RGB}{50,205,50}
\definecolor{yellow}{RGB}{240,240,10}

\newcommand{\Gthreeand}{\G_{3'}}
\newcommand{\Gthreeor}{\G_{3''}}
\newcommand{\Tthreeand}{T_{3'}}
\newcommand{\Tthreeor}{T_{3''}}

\title{Graphs generated from minimal sets of finite point-set topologies}

\author{Ketai Chen}\thanks{K. Chen: Department of Mathematics and Statistics, Auburn University.  \url{kzc0036@auburn.edu}}
\author{Jared DeLeo}\thanks{J. DeLeo: Department of Mathematics and Statistics, Auburn University.  \url{jmd0150@auburn.edu}}
\author{Owen Henderschedt}\thanks{O. Henderschedt: Department of Mathematics and Statistics, Auburn University.  \url{olh0011@auburn.edu}}

\begin{document}
\begin{abstract}
In 1985, Golumbic and Scheinerman established an equivalence between comparability graphs and containment graphs—graphs whose vertices represent sets, with edges indicating set containment. A few years earlier, McMorris and Zaslavsky characterized upper bound graphs—those derived from partially ordered sets where two elements share an edge if they have a common upper bound—by a specific edge clique cover condition. In this paper, we introduce a unifying framework for these results using finite point-set topologies. Given a finite topology, we define a graph whose vertices correspond to its elements, with edges determined by intersections of their minimal containing sets, where intersection is understood in terms of the topological separation axioms. This construction yields a natural sequence of graph classes—one for each separation axiom—that connects and extends both classical results in a structured and intuitive way.
\end{abstract}

\maketitle

\section{Introduction}

Given a partially ordered set (poset) $P = (X,\leq)$, one can construct many graphs which have been of interest in extremal set theory and graph theory. Among the most famous is the \textit{comparability graph} of $P$ which is the graph whose vertices are $X$ and $x,y\in X$ are adjacent if $x\leq y$ or $y\leq x$. In general, we say a graph $G = (V,E)$ is a comparability graph if it arises from some poset $P = (V,\leq)$. There are many graph-theoretic characterizations for comparability graphs including graphs which are \textit{transitively orientable}, that is, there exists an orientation of $G$ such that if $x\to y$ and $y\to z$, then $x\to z$. This orientation can be easily obtained from the comparability graph of a poset by orienting $x\to y$ if $x\leq y$. Comparability graphs are both subclasses and superclasses of other well-known and actively researched graph classes. Subclasses include complete graphs, bipartite graphs, permutation graphs, threshold graphs, and co-graphs ($P_4$-free); see \cite{THRESHHOLD,PERMUTATION}. Superclasses include perfect graphs and word-representable graphs; see \cite{PERFECT-SURVEY,WORD-REP}.

In what may seem like a different context, a graph $G=(V,E)$ is a \textit{$\Sigma$-containment graph} if there exists a family of sets $\Sigma$ such that for each $v_i\in V(G)$, there exists a corresponding $S_i\in \Sigma$ such that $v_iv_j\in E(G)$ if and only if $S_i\subset S_j$ or $S_j \subset S_i$. In \cite{Containment}, M. Golumbic and E. Scheinerman proved that the two classes of graphs are, in fact, equivalent.

\begin{theorem}[\cite{Containment}]\label{thm: cont=comp}
A graph is a comparability graph if and only if it is a containment graph.
\end{theorem}

To see one direction of this theorem, for each $S_i,S_j\in \Sigma$ with $S_i$ and $S_j$ adjacent in $G$, we let $S_i\leq S_j$ if $S_i\subset S_j$, or, in terms of an orientation we can orient $S_i\to S_j$. To see the other direction of this characterization, one can construct a set system $\Sigma$ by associating for each element $x$ of a poset $P =(X,\leq)$, the set $S(x) = \{y:y\leq x\}$.

Another lesser-known graph obtained from a poset $P=(X,\leq)$ is the \textit{upper bound graph of $P$}. Like the comparability graph of $P$, this graph also has vertex set $X$ but $x,y\in X$ are adjacent if and only if there exists some $z\in X$ with $x\leq z$ and $y\leq z$. If $x \leq y$, i.e. the edge $xy$ exists in the comparability graph of $P$, then $y$ satisfies the conditions for $z$, as $y$ is a common upper bound for $x$ and $y$. Hence, the edges in an upper bound graph of $P$ can be obtained by taking the edges of the comparability graph of $P$ and adding edges. This notion of additional edges will be at the heart of our paper. 

Once again, in what may seem like a shift in direction, given a graph $G = (V,E)$, we say $v\in V(G)$ is a \textit{simplicial vertex} if its closed neighborhood, $N[v]$, induces a clique (a complete subgraph). We say a graph $G = (V,E)$ is an \textit{edge-simplicial} graph if there exists a collection of cliques that cover the edges of $G$ such that every clique contains a simplicial vertex. Note that this definition of edge-simplicial graphs is equivalent to requiring every vertex to either be a simplicial vertex or to have a neighbor which is a simplicial vertex. These graphs are also equivalent to the class of graphs whose clique edge-cover number equals its clique vertex-cover number. These parameters are defined as the minimum number of cliques in a graph required to cover every edge (vertex). We will discuss edge-simplicial graphs in more detail in \cref{sec:More with T2}. In \cite{UPPERBOUND}, McMorris and Zaslavsky in fact, showed that the class of edge-simplicial graphs are equivalent to upper bound graphs.

\begin{theorem}[\cite{UPPERBOUND}]\label{thm:upperbound}
A graph is an upper bound graph if and only if it is an edge-simplicial graph.   
\end{theorem}

The goal of this paper is to unify these classical results and extend them both in terms of graphs generated from a poset and graphs with a clique-like structure (as seen with edge-simplicial graphs). To do this, we will soon define a sequence of six graph classes \[\G_0,\hspace{.1cm}\G_1,\hspace{.1cm}\G_2,\hspace{.1cm}\Gthreeand,\hspace{.1cm} \Gthreeor, \text{ and }\hspace{.1cm} \G_4\]
whose definitions we deliberately hold off on giving until \cref{sec:building-graphs}. These graph classes and the framework used to extend the notions of \cref{thm: cont=comp} and \cref{thm:upperbound} are built using finite topologies, which will require some further definitions. Before giving those, we first extend the notion of comparability graphs and upper bound graphs by introducing three additional new graphs obtained from a poset. 

Let $P=(X,\leq)$ be a poset. The \textit{half-closed upper bound graph of $P$} is the graph whose vertices are $X$ and $x,y\in X$ are adjacent if and only if there exists $a,b\in X$ such that $a\leq x$, $a\leq b$, and $y\leq b$ \textit{and} there exists $\alpha,\beta\in X$ such that $\alpha\leq y$, $\alpha\leq \beta$, and $x\leq \beta$. In words, to be adjacent, two elements of a poset do not necessarily need to have a common upper bound anymore, but rather they each need a lower bound ($a$, or $\alpha$) which together with $x$ or $y$ have a common upper bound ($b$ or $\beta$). If we relax the requirement of half-closed upper bound graphs to require either that there exists $a,b\in X$ such that $a\leq x$, $a\leq b$, and $y\leq b$ \textit{or} there exists $\alpha,\beta\in X$ such that $\alpha\leq y$, $\alpha\leq \beta$, and $x\leq \beta$, then we obtain the \textit{fully-closed upper bound graph of $P$.} Finally, the \textit{extended-closed upper bound graph} of $P$ is the graph whose vertices are $X$ and $x,y\in X$ are adjacent if and only if there exists $a,b,c\in X$ such that $a\leq x$, $b\leq y$, $a\leq c$ and $b\leq c$.

Given a poset $P=(X,\leq)$, as we have mentioned, the upper bound graph of $P$ can be obtained from the comparability graph of $P$ by adding additional edges. Likewise, the half-closed upper bound graph can be obtained from the upper bound graph by adding additional edges, the same process holds going to the fully-closed upper bound graph of $P$ and then the extended-closed upper bound graph of $P$. In \cref{sec: comp proof}, we prove the following.

\begin{theorem}\label{thm:poset-characterization}
The following holds.
\begin{enumerate}
    \item A graph $G\in \G_0$ if and only if $G$ is a disjoint union of cliques.
    \item A graph $G\in \G_1$ if and only if $G$ is a comparability graph.
    \item A graph $G\in \G_2$ if and only if $G$ is an upper bound graph.
    \item A graph $G\in \Gthreeand$ if and only if $G$ is a half-closed upper bound graph.
    \item A graph $G\in \Gthreeor$ if and only if $G$ is a fully-closed upper bound graph.
    \item A graph $G\in \G_4$ if and only if $G$ is an extended-closed upper bound graph.
\end{enumerate}
\end{theorem}

By \cref{thm:upperbound} and \cref{thm:poset-characterization}, we see that $\G_2$ and edge-simplicial graphs are equivalent. In \cref{sec:More with T2}, we take a closer look at the graph class $\G_2$ (edge-simplicial and upper bound graphs) and fully characterize the \textit{critical graphs} of $\G_2$. By critical, we mean a graph $G\in \G_2$, but the removal of any edge set $S$ either disconnects the graph or results in $G-S\notin \G_2$. We define $\G_2^*$ as the graphs in $\G_2$ that are critical. Given a graph $G$, \textit{anchoring a star to} $G$ is the operation of adding a $K_{1,t}$ and joining the vertex in the part of size $1$ to a maximal clique of $G$. This leads to the following theorem.

\begin{theorem}\label{thm:critical-new}
$G\in \G_2^*$ if and only if $G\cong K_2$ or $G$ is connected and can be obtained by iteratively anchoring stars to $K_1$.   
\end{theorem}

One way to interpret edge-simplicial graphs is as graphs covered by a collection of cliques, each containing a distinguished vertex—a simplicial vertex. This raises the question of whether other graph classes, such as $\Gthreeand$, $\Gthreeor$, and $\G_4$, can also be characterized in terms of clique coverings. We take a step in this direction by establishing an analogous result for $\Gthreeand$. Instead of cliques, we define a collection of subgraphs with a clique-like structure, which we call \emph{universes} (see \cref{sec:G3and}). The role of simplicial vertices in this setting is replaced by a new type of distinguished vertex, which we call \emph{suns}, as they play a central role in the universe structure. We give the necessary definitions and prove the following theorem in \cref{sec:G3and}.

\begin{theorem}\label{thm:G3and-universe}
$G\in \Gthreeand$ if and only if $G$ has an edge covering with universes $\mathcal{U} = U_1,U_2,...,U_t$ where $U_i = (s_i,P_i,M_i)$, $P = \bigcup_{i=1}^t P_i$, and $S = \{s_1,...,s_t\}$ and the following hold.

\begin{enumerate}
   \item $S\cup P$ partition $V(G)$.
   \item $\forall s_i,s_j \in S$ with $N(s_i) \cap N(s_j) \cap P \not= \emptyset$, then $s_is_j \in E(G)$.
\end{enumerate}
\end{theorem}

We now outline the structure of the paper. In \cref{sec:building-graphs}, we define the six graph classes that extend the results of \cref{thm: cont=comp} and \cref{thm:upperbound}. This includes establishing the necessary topological definitions, and motivating  the adjacency between vertices which are interpreted via topological separation axioms. In \cref{sec:prelim}, we present key preliminary observations and lemmas that aid in characterizing these graph classes. In \cref{sec: comp proof}, we prove \cref{thm:poset-characterization}. Next, in Section \ref{sec:More with T2}, we take a closer look at the graph class $\G_2$. Among others, we prove \cref{thm:critical-new}. Finally, in \cref{sec:G3and}, we examine some properties of the graph classes, $\Gthreeand$ and prove \cref{thm:G3and-universe}. For terms not defined here, see \cite{west}.

\section{Building the graph classes $\G_0, \G_1, \G_2, \Gthreeand,  \Gthreeor, \text{ and } \G_4$}\label{sec:building-graphs}

We begin with the objects of interest for the remainder of the paper.

\begin{definition}
A \textit{(finite) topological space} $(X,\tau)$ is a pair where $X$ is a (finite) ground set and $\tau$ is a \textit{(finite) topology} or a family of subsets of $X$ such that \begin{enumerate}
    \item $\emptyset, X\in \tau$.
    \item $\tau$ is closed under unions.
    \item $\tau$ is closed under finite intersections.
\end{enumerate}
\end{definition}

If $S\subseteq X$, the collection $\tau_S = \{ S\cap U : U\in \tau\} $ is a topology on $S$, called a \textit{subspace topology}. We say $K \subseteq X$ is a \textit{closed set of $\tau$} if $X\setminus K\in \tau$. Given a topology $\tau$ on $X$ our goal is to construct a graph whose vertex set is $X$. Naturally, we want elements of $X$ to be \textit{disconnected} if they are \textit{separated} within the topology $\tau$. The notion of separation has been well studied in topology, from which the following definition arises. For convenience, throughout this paper, when we say $i = 3$, we mean the relevant statement holds for both $3'$ and $3''$. 

\begin{definition}\label{def: sep-axioms}
    Let $i\in \{0,1,2,3,4\}$. Let $\tau$ be a topology on $X$ and let $x,y\in X$. We say $x$ and $y$ are $T_i$-separated if the following holds.
    
    \begin{enumerate}
    \setlength{\itemsep}{.25em}
        \item \textit{$T_0$-separated:} If there exists a $U\in \tau$ with either $x\in U$ and $y\notin U$, or $x\notin U$ and $y\in U$. 
        \item \textit{$T_1$-separated:} If there exists $U_x,U_y\in \tau$ with $x\in U_x$, $y\notin U_x$, and $y\in U_y$,  $x\notin U_y$.
        \item \textit{$T_2$-separated:} If there exists $U_x, U_y\in \tau$ with $x\in U_x$, $y\in U_y$, and $U_x\cap U_y = \emptyset$.
        \item \textit{$\Tthreeand$-separated:} If there exists a closed set $J_x$ of $\tau$ and open sets $U_J,U_y \in \tau$ with $x\in J_x\subseteq U_J$, $y\in U_y$, and $U_J\cap U_y = \emptyset$, \textit{or} if there exists a closed set $K_y$ of $\tau$ and open sets $U_x,U_K \in \tau$ with $x\in U_x$, $y\in K_y\subseteq U_K$ and $U_x\cap U_K = \emptyset$.
        \item \textit{$\Tthreeor$-separated:} If there exists a closed set $J_x$ of $\tau$ and open sets $U_J,U_y \in \tau$ with $x\in J_x\subseteq U_J$, $y\in U_y$, and $U_J\cap U_y = \emptyset$, \textit{and} if there exists a closed set $K_y$ of $\tau$ and open sets $U_x,U_K \in \tau$ with $x\in U_x$, $y\in K_y\subseteq U_K$ and $U_x\cap U_K = \emptyset$.
        \item \textit{$T_4$-separated:} If there exists closed sets $J_x,K_y$ of $\tau$ and open sets $U_J,U_K \in \tau$ with $x\in J_x\subseteq U_J$ and $y\in K_y\subseteq U_K$, and $U_J\cap U_K = \emptyset$.
    \end{enumerate}
\end{definition}

\begin{figure}[h!]
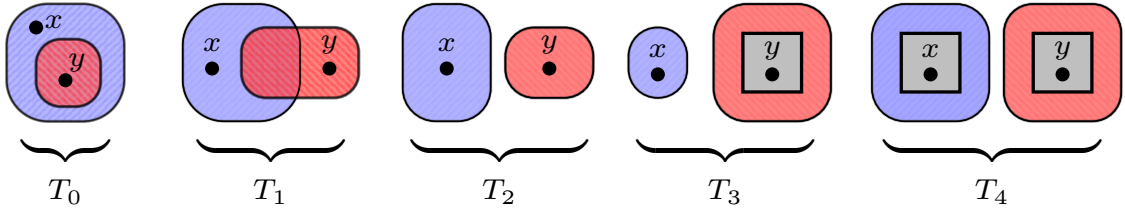

    \centering
    \includestandalone[width=.9\textwidth]{figures/separation-axioms}
    \caption{Visual interpretation of $x$ and $y$ being $T_i$-separated in \cref{def: sep-axioms}. Open sets in $\tau$ indicated in red and blue, and closed sets are indicated in gray.}
    \label{fig:sep-axioms}
\end{figure}

The definition of $T_i$-separation is motivated by the separation axiom in topology. For a visual interpretation of \cref{def: sep-axioms}, see \cref{fig:sep-axioms}, and for more on the separation axioms within current research, see the survey paper \cite{TopAxioms}. We now are ready to define the graphs of interest in this paper. Given a topology $\tau$ on $X$, we can intuitively imagine starting with a complete graph $K_{|X|}$ and removing edges between $x, y \in X$ when $x$ and $y$ are $T_i$-separated with respect to $\tau$.

\begin{definition}\label{def: Ti-sep}
Let $\tau$ be a topology on $X$ and let $i\in \{0,1,2,3,4\}$. $G_i(\tau)$ is the graph whose vertices are $X$ and $x,y\in X$ are \textit{not} adjacent if they are $T_i$-separated (see \cref{def: sep-axioms}). 
\end{definition}

\begin{figure}[h!]
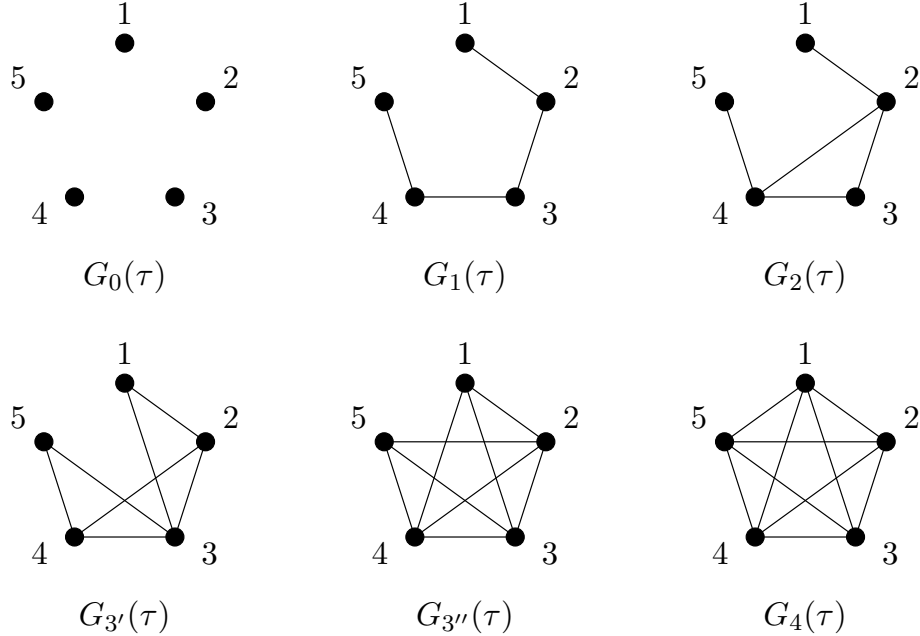

    \centering
    \includestandalone[width=.75\textwidth]{figures/Ti-examples}
    \caption{The six graphs corresponding to the topology $\tau$ on $[5]$ with minimal base $M(\tau) = \{\{1\},\{1,2,3\},\{3\},\{3,4,5\},\{5\}\}$.}
    \label{fig: 5 graphs example}
\end{figure}

One motivating reason for defining the graphs $G_i(\tau)$ in this way is that for a fixed topology $\tau$ on $X$ with $x,y\in X$, and $i\geq 1$, if $x$ and $y$ are $T_i$-separated then they are $T_{i-1}$-separated. In other words, \[E(G_0(\tau))\subseteq E(G_1(\tau))\subseteq E(G_2(\tau))\subseteq E(G_{3'}(\tau))\subseteq E(G_{3''}(\tau))\subseteq E(G_4(\tau)).\] Given two vertices $x$ and $y$ in a graph, it is often more natural to ask when are $x$ and $y$ adjacent, rather than ask when they are \textit{not} adjacent. When $\tau$ is a finite topology on $X$, to determine whether or not $x,y\in X$ are $T_i$-separated, it suffices to consider \textit{minimal sets}. For $x\in X$, the \textit{minimal set of $\tau$ containing $x$} is defined as $m_\tau(x) = \bigcap_{U\in \tau: x\in U} U$. We let $M(\tau) = \{U\in \tau: U=m_{\tau}(x) \text{ for some $x\in X$}\}$ be the \textit{minimal base} for $\tau$. We can also extend the definition of $m_{\tau}$ for subsets $S\subseteq X$ as $m_{\tau}(S) = \bigcap_{U\in \tau, S\subseteq U} U$. Given $x\in X$, we define $\overline{\ell}_{\tau}(x) = \bigcup_{U\in \tau, x\notin U}U$ and $\ell_{\tau}(x) = X\setminus \overline{\ell}_{\tau}(x)$. When $S=\ell_\tau(x)$ for some $x\in X$, for convenience, we write $m_{\tau}(\ell(x))$. We now have the following equivalent definitions for $x,y\in X$ \textit{not} being $T_i$-separated.

\begin{definition}
Let $\tau$ be a finite topology on $X$ and let $x,y\in X$. We say $x$ and $y$ are
\begin{enumerate}
    \item \textit{$T_0$-adjacent} if $m_{\tau}(x) = m_{\tau}(y)$.
    \item \textit{$T_1$-adjacent} if $m_{\tau}(x)\subseteq m_{\tau}(y)$ or $m_{\tau}(y)\subseteq m_{\tau}(x)$.
    \item \textit{$T_2$-adjacent} if $m_{\tau}(x)\cap m_{\tau}(y) \neq \emptyset$.
    \item \textit{$\Tthreeand$-adjacent} if $m_{\tau}(x) \cap m_{\tau}(\ell(y)) \neq \emptyset$ and
 $m_{\tau}(\ell(x)) \cap m_{\tau}(y) \neq \emptyset$.
 \item \textit{$\Tthreeor$-adjacent} if $m_{\tau}(x) \cap m_{\tau}(\ell(y)) \neq \emptyset$ or
 $m_{\tau}(\ell(x)) \cap m_{\tau}(y) \neq \emptyset$.
        \item \textit{$T_4$-adjacent} if $m_{\tau}(\ell(x))\cap m_{\tau}(\ell(y)) \neq \emptyset$.
\end{enumerate}
\end{definition}
In \cref{sec:prelim} (\cref{lem: sep = adj}) we explicitly show that for finite topologies \[\text{$x,y$ are $T_i$-adjacent} \iff \text{$x,y$ are not $T_i$-separated.} \]

For an example of a finite topology $\tau$ and its six corresponding graphs, see \cref{fig: 5 graphs example}. Throughout this paper, we let \[\G_i = \{G : G\cong G_i(\tau) \text{ for some finite topology $\tau$}\}.\] It should be noted that in this paper only finite topologies are considered. Infinite graphs that arise from an infinite topology, and their subsequent characterizations, remains an open problem. One of our motivating questions is to fully characterize the graphs that are in $\G_i$. For example, it is fairly easy to see that $G\in \G_0$ if and only if $G$ is a disjoint union of cliques (for completeness we will give a proof of this in \cref{sec:prelim}). However, the other graph classes, $\G_i$ for $i\in [4]$, reveal more intricate structures.

\section{Preliminary lemmas and observations}\label{sec:prelim}

Before characterizing graphs in $\G_i$, we prove several important statements regarding finite topologies and the corresponding graphs constructed by them.

\begin{lemma}\label{lem: sep = adj}
Given a finite topology $\tau$ on $X$ and $x,y\in X$, for all $i\in \{0,1,2,3,4\}$, $x,y$ are $T_i$-adjacent if and only if they are not $T_i$-separated.
\end{lemma}

\begin{proof}
First we show that if $x, y$ are $T_i$-separated then $x,y$ are not $T_i$-adjacent. Suppose $x, y$ are $T_0$-separated. Then up to symmetry, there exists a $U\in \tau$ with $x\in U$ and $y\notin U$. But then since $m_\tau(x)\subseteq U$, and $y\notin m_{\tau}(x)$, this implies that $m_{\tau}(x)\neq m_{\tau}(y)$. Thus, $x,y$ are not $T_0$-adjacent. 

Next suppose $x, y$ are $T_1$-separated. Then there exists $U_x,U_y\in \tau$ with $x\in U_x$, $y\notin U_x$, and $y\in U_y$,  $x\notin U_y$. But since $x\in m_{\tau}(x)\subseteq U_x$ and $y\in m_{\tau}(y)\subseteq U_y$, $m_{\tau}(x)\not \subseteq m_{\tau}(y)$ and $m_{\tau}(y)\not \subseteq m_{\tau}(x)$. Thus, $x,y$ are not $T_1$-adjacent. 

Suppose $x, y$ are $T_2$-separated. Then there exists $U_x, U_y\in \tau$ with $x\in U_x$, $y\in U_y$, and $U_x\cap U_y = \emptyset$. Since $m_{\tau}(x)\subseteq U_x$ and $m_{\tau}(y)\subseteq U_y$, $m_{\tau}(x)\cap m_{\tau}(y) = \emptyset$. Thus, $x$ and $y$ are not $T_2$-adjacent. 

Suppose $x,y$ are $\Tthreeand$-separated. Then up to symmetry, there exists a closed set $J_x$ of $\tau$ and open sets $U_J,U_y\in \tau$ with $x\in J_x\subseteq U_J$, $y\in U_y$, and $U_J\cap U_y = \emptyset$. Since $x\in \ell(x)\subseteq J_x\subseteq U_J$ and $y\in U_y$, $m_\tau (\ell (x))\subseteq U_J$ and $m_\tau (y)\subseteq U_y$. So $m_\tau (\ell (x))\cap m_\tau (y)=\emptyset $. Thus, $x,y$ are not $\Tthreeand$-adjacent.

Suppose $x,y$ are $\Tthreeor$-separated. Then if there exist a closed set $J_x$ and open sets $U_J,U_y$ in $\tau$ with $x\in J_x\subseteq U_J$, $y\in U_y$, and $U_J\cap U_y = \emptyset$, \textit{and} if there exist closed sets $K_y$ and open sets $U_x,U_K$ of $\tau$ with $x\in U_x$, $y\in K_y\subseteq U_K$ and $U_x\cap U_K = \emptyset$.
Since $x\in \ell(x)\subseteq J_x\subseteq U_J$ and $y\in U_y$, $m_\tau (\ell (x))\subseteq U_J$ and $m_\tau (y)\subseteq U_y$. So $m_\tau (\ell (x))\cap m_\tau (y)=\emptyset $.
And since $y\in \ell(y)\subseteq K_y\subseteq U_K$ and $x\in U_x$, $m_\tau (\ell (y))\subseteq U_K$ and $m_\tau (x)\subseteq U_x$. So $m_\tau (\ell (y))\cap m_\tau (x)=\emptyset $. Thus, $x,y$ are not $\Tthreeor$-adjacent.

Suppose $x,y$ are $T_4$-separated. Then there exist closed sets $J_x,K_y$ and open sets $U_J,U_K$ of $\tau$ with $x\in J_x\subseteq U_J$ and $y\in K_y\subseteq U_K$, and $U_J\cap U_K = \emptyset$. Since $x\in \ell(x)\subseteq J_x$ and $y\in \ell(y)\subseteq K_y$, $m_\tau (\ell (x))\subseteq U_J$ and $m_\tau (\ell (y))\subseteq U_K$. So $m_\tau (\ell (x))\cap m_\tau (\ell (y))=\emptyset $. Thus, $x,y$ are not $T_4$-adjacent.

The other direction also holds since if $x$ and $y$ are not $T_i$-adjacent then $m_{\tau}(x)$ and $m_{\tau}(y)$ are the sets in $\tau$ which are forbidden in the definition of $T_i$-separation.
\end{proof}

\begin{proposition}\label{prop:G0}
 $G\in \G_0$ if and only if $G$ is a disjoint union of cliques.
\end{proposition}

\begin{proof}
If $G$ is disconnected, the same argument holds for each component, so we may assume $G$ is connected. It is easy to see that any complete graph is in $\G_0$ by letting $\tau = \{\emptyset, V(G)\}$. Let $G \in \G_0$ and assume for contradiction that $G$ is connected but not complete. Then there exists a vertex $x$ which is not adjacent to some vertex $y$. Since they are not adjacent, without loss of generality, we can assume there is an open set $U_x$ such that $x\in U_x$ and $y\notin U_x$. Then every vertex in $U_x$ is not adjacent to every vertex that is not in $U_x$, which contradicts the connectedness of $G$.
\end{proof}

The following observations will be quite useful when handling minimal sets in finite topologies.

\begin{observation}\label{obs: topology}
Let $\tau$ be a finite topology on $X$ and let $x,y,z\in X$. The following hold.
\begin{enumerate}
    \item\label{obs:intersection} If $y\in m_{\tau}(x)$ then $m_{\tau}(y)\subseteq m_{\tau}(x)$.
    \item\label{obs: ell intersection} $x\in \ell_{\tau}(y)$ if and only if $m_{\tau}(y)\subseteq m_{\tau}(x)$.
    \item If $m_{\tau}(x)\subset m_{\tau}(y)$ then $x\notin \ell_{\tau}(y)$ and $m_{\tau}(x)\subseteq m_{\tau}(\ell(y))$.
    \item\label{obs: ell must exist} If $m_{\tau}(\ell(x))\cap m_{\tau}(y)\neq \emptyset$, then there exists some $w\in \ell_{\tau}(x)$ such that $m_{\tau}(w)\cap m_{\tau}(y)\neq \emptyset$.
    \item\label{obs:T3-and} If $x,y\in m_{\tau}(z)$, then $m_{\tau}(\ell(x))\cap m_\tau(y)\neq \emptyset$ and $m_{\tau}(x)\cap m_{\tau}(\ell(y))\neq \emptyset$.
\end{enumerate}
\end{observation}

\begin{proof}
We first show (1). Suppose for the contradiction that $m_{\tau}(y)\not \subseteq m_{\tau}(x)$, then $m_{\tau}(y) \cap m_{\tau}(x)$ is an open set containing $y$ since $y$ is in $m_{\tau}(x)$, which is contained in $m_{\tau}(x)$, and which contradicts to the minimality of $m_{\tau}(y)$.

Next we prove (2). If $x\in \ell_{\tau}(y)$, then by definition, $x\notin \overline{\ell}_{\tau}(y)$. But this means that any set containing $x$ must contain $y$. Hence, $m_{\tau}(y)\subseteq m_{\tau}(x)$. If $m_{\tau}(y)\subseteq m_{\tau}(x)$, then $x\notin \overline{\ell}_{\tau}(y)$ which implies $x\in \ell_{\tau}(y)$.

Next we prove (3). If $m_{\tau}(x)$ is a strict subset of $m_{\tau}(y)$ then there exists a set $A\in \tau$ with $x\in A$ and $y\notin A$. Hence, $x\in \overline{\ell}_{\tau}(y)$ which implies $x\notin \ell_{\tau}(y)$. We also have $m_{\tau}(x)\subset m_{\tau}(y)\subseteq m_{\tau}(\ell(y))$.

Next we prove (4). It suffices to show that $m_{\tau}(\ell(x)) = \bigcup_{a\in \ell_\tau(x)} m_{\tau}(a)$. Certainly $m_{\tau}(\ell(x))\subseteq \bigcup_{a\in \ell_{\tau}(x)}m_{\tau}(a)$ so we now have to show that there is no smaller set in $\tau$ containing $\ell(x)$. This is true because for any $b\in \bigcup_{a\in \ell_{\tau}(x)}m_{\tau}(a)$, $b\in m_{\tau}(a)\subseteq m_{\tau}(\ell(x))$ for some $a\in \ell_{\tau}(x)$ and thus, $\bigcup_{a\in \ell_{\tau}(a)}m_{\tau}(a)\subseteq m_{\tau}(\ell(x))$.

Lastly we prove (5). If $x,y\in m_{\tau}(z)$, then by (1), $m_{\tau}(x)\subseteq m_{\tau}(z)$ and $m_{\tau}(y)\subseteq m_{\tau}(z)$. But then by (2), $z\in \ell_{\tau}(x)\cap \ell_{\tau}(y)$. Thus, $m_{\tau}(\ell(x))\cap m_{\tau}(y)\neq \emptyset$ and $m_{\tau}(x)\cap m_{\tau}(\ell(y))\neq \emptyset$.
\end{proof}

A given graph $G\in \G_i$ may not necessarily have a unique finite topology associated with $G$. For example, let $X = \{1,2,3\}$, let $\tau = \{\emptyset,\{1\}, \{1,2\},\{1,2,3\}\}$, let $\tau' = \{\emptyset,\{1,2,3\}\}$, and let $\tau'' = \{\emptyset,\{1,2\},\{2\}, \{2,3\}, \{1,2,3\}\}$. In all three cases, the $G_2(\tau)$, $G_2(\tau')$, and $G_2(\tau'')$ are all isomorphic to $K_3$. Because of this, in analyzing a graph $G$ it will be helpful to assume we are working with a topology with certain properties. To this end, we prove the following three lemmas.

\begin{lemma}\label{lem:m(x)=m(y)}
Let $i\in \{1,2,3,4\}$. For every $G \in \G_i$, there exists a finite topology $\tau$ on $V(G)$ such that $G_i(\tau)$ is isomorphic to $G$ and $m_{\tau}(x)\neq m_{\tau}(y)$ for all $x,y \in V(G)$.
\end{lemma}

\begin{proof}
Let $G= G_i(\tau)$ be a graph with associated topology $\tau$ and suppose for contradiction that there is no topology $\tau'$ such that the graph $G':=G_i(\tau')$ is isomorphic to $G$ and $m_{\tau'}(x)\neq m_{\tau'}(y)$ for all $x,y\in V(G')$. Of all counterexamples, choose such a $G$ and $\tau$ such that $\tau$ has the fewest pairs $x,y$ with $m_{\tau}(x)=m_{\tau}(y)$. 

Now consider the new topology $\tau'$ generated by all the open sets in $\tau$ and the open set $m_\tau (x)\backslash \{y\}$ (see \cref{fig:adj-top}). Clearly $m_{\tau'} (x)=m_{\tau} (x)\backslash \{y\}$. So $\tau'=\tau \cup \{U\cup m_{\tau'} (x) : U\in \tau \}$.
Let $\tau^d=\tau'\backslash \tau$, which is the collection of open sets that appear in $\tau'$ but not in $\tau$. Let $S_x=\{a\in V(G): m_\tau (a)=m_\tau (x) \} $.

First, notice 
\begin{equation}
m_{\tau'}(v) = \begin{cases}
               m_{\tau}(v)\setminus \{y\} &\text{ for all $v\in S_x\setminus \{y\}$} \\ m_{\tau}(v) &\text{ for all $v\notin S_x\setminus \{y\}$}
    \end{cases}
\end{equation}

and \begin{equation}\ell_{\tau'}(v)=\begin{cases}
         \ell_{\tau}(v)\setminus m_{\tau'}(x)  &\text{ for $v = y$}\\
        \ell_{\tau}(v) &\text{ for all $v\neq y$}.
    \end{cases}
    \end{equation}
    
To complete the proof, it suffices to show that $E(G_{i}(\tau)) = E(G_i(\tau'))$ which would contradict our choice of $\tau$. It easy to see that if $uv\notin E(G_i(\tau))$ then $uv\notin E(G_i(\tau'))$, so it suffices to show that if $uv\in E(G_i(\tau))$, then $uv\in E(G_i(\tau'))$.

To help in showing this, we first prove the following claim. \begin{claim}\label{claim: min-sets}
    $m_{\tau'}(\ell(v)) = m_{\tau}(\ell(v))$ for all $v\in V(G)$.
\end{claim}

\begin{proofc}
    If $v\notin m_\tau(x)$, then because of (2) we have $\l_{\tau'}(v)=\l_{\tau}(v)$. And since $\l_{\tau'}(v) \cap m_{\tau'}(x)=\emptyset$, $m_\tau (\l(v))\notin \tau^d$. Hence, $m_{\tau'}(\l(v))=m_{\tau}(\l(v))$. If $v=y$, by (2), we have $\l_{\tau'}(y)\cap m_{\tau'}(x)=\emptyset$, then $m_{\tau'}(\l(y))\notin\tau^d$. Thus, $m_{\tau'}(\l_(y))=m_{\tau}(\l_(y))$. If $e\in m_{\tau'}(x)$, by (2), we have $\l_{\tau'}(v)=\l_{\tau}(v)$. According to \cref{obs: topology}(\ref{obs:intersection}) and \cref{obs: topology}(\ref{obs: ell intersection}) we have $y\in \ell_{\tau'}(v)$. Hence, $y\in m_{\tau'}(\ell(v))$ which again implies $m_{\tau'}(\ell(v))\notin \tau^d$. Thus, $m_{\tau'}(\l_(v))=m_{\tau}(\l_(v))$.
\end{proofc}

We now use (1) and (2) along with \cref{claim: min-sets} to prove that for any $uv\in E(G_i(\tau))$, $uv\in E(G_i(\tau'))$.

If $i=4$, the statement holds by directly applying \cref{claim: min-sets}. If $i=3$, then it suffices to show that if $m_{\tau}(u)\cap m_{\tau}(\l(v))\neq \emptyset$, then $m_{\tau'}(u)\cap m_{\tau'}(\l(v))\neq \emptyset$. If there exist $z\in m_{\tau}(u)\cap m_{\tau}(\l(v))$ with $z\neq y$, then by (1) and \cref{claim: min-sets}, we have $z\in m_{\tau'}(u)\cap m_{\tau'}(\l(v))$. But if $m_{\tau}(u)\cap m_{\tau}(\l(v))=\{y\}$, then $\{y\}\in \tau$, which cannot happen because $m_\tau(x)=m_\tau(y)$. The case when $i = 2$, follows the same logic as the previous case with $m_{\tau}(u)$ and $m_{\tau}(v)$ instead of $m_{\tau}(u)$ and $m_{\tau}(\ell(v))$. Finally, if $i=1$, then without lose of generality, let $m_\tau(u)\subseteq m_\tau(v)$. If $y\notin m_\tau(u)$ then $m_{\tau'}(u)\subseteq m_{\tau'}(v)$ by (1). If $y\in m_\tau(u)$ then $y\in m_\tau(v)$, and then $m_{\tau'}(u)\subseteq m_{\tau'}(v)$.

 Thus, $E(G_i(\tau))=E(G_i(\tau'))$.
\end{proof}

\begin{figure}[h!]
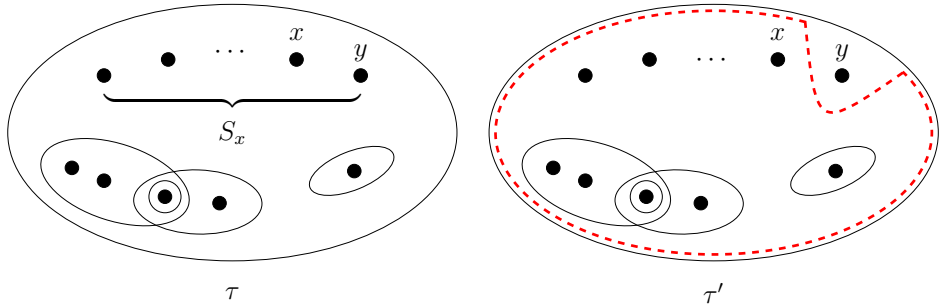

    \centering
    \includestandalone[width=0.75\textwidth]{figures/adjusting-top-1}
    \caption{An example of $m_{\tau}$ and $m_{\tau'}$ for \cref{lem:m(x)=m(y)}}
    \label{fig:adj-top}
\end{figure}

Notice that for a given graph $G$, the proof of \cref{lem:m(x)=m(y)} provides us with a way to find such a topology $\tau$ in which $G \cong G_i(\tau)$ and $m_{\tau}(x)\neq m_{\tau}(y)$ for all $x,y \in V(G)$. This is done by repeatedly adding the sets $m_{\tau}(x) \setminus\{y\}$ to $\tau$ until $m_{\tau}(x)\neq m_{\tau}(y)$ for all $x,y\in V(G)$.

Given a finite topology $\tau$ on $X$, the \textit{height} of $\tau$ is the largest integer $h$ in which there exists distinct elements, $v_1,...,v_h\in X$, with $m_{\tau}(v_1)\subseteq m_{\tau}(v_2) \subseteq \cdots \subseteq m_{\tau}(v_h)$. We say $(v_1,...,v_h)$ is an $h$-chain in $\tau$.

\begin{lemma}\label{lem: height2}
    Let $\tau$ be a finite topology on $X$. Then there exists a finite topology $\tau'$ on $X$ such that $G_2(\tau) \cong G_2(\tau')$ and the height of $\tau'$ is at most $2$.
\end{lemma}

\begin{proof}
Suppose the theorem is false and let $\tau$ be a finite topology on $X$ which is a minimum counterexample with respect to the height $h$ of $\tau$, subject to the number of distinct $3$-chains in $\tau$. If the height of $\tau$ is at most $2$ then we are done. Suppose the height of $\tau$ is at least $3$. Let $(v_1,v_2,v_3)$ be a $3$-chain in $\tau$.

We now form $\tau'$ from $\tau$ by simply removing $v_2$ from any set containing $v_3$. Thus, $(v_1,v_2,v_3)$ is no longer a $3$-chain in $\tau'$. Since the height of $\tau'$ is at most the height of $\tau$, what remains to show is that $G_2(\tau)\cong G_2(\tau')$, contradicting the choice of $\tau$. 
As sets only become more separated in $\tau'$, we must only check that for any $uv\in E(G_2(\tau))$, $uv\in E(G_2(\tau'))$. If $x\in m_\tau(u)\cap m_{\tau}(v)$ and $x\neq v_2$, then $x\in m_{\tau'}(u)\cap m_{\tau'}(v)$. If $v_2\in m_{\tau}(u)\cap m_{\tau}(v)$, then $m_{\tau}(v_2)\subseteq m_{\tau}(u)$ and $m_{\tau}(v_2)\subseteq m_{\tau}(v)$. Since $m_{\tau}(v_1)\subseteq m_{\tau}(v_2)$, by definition of our $3$-chain, $v_1\in m_{\tau}(u)\cap m_{\tau}(v)$ and hence $v_1\in m_{\tau'}(u)\cap m_{\tau'}(v)$. Thus, $uv\in E(G_2(\tau'))$.
\end{proof}

\begin{figure}[h!]
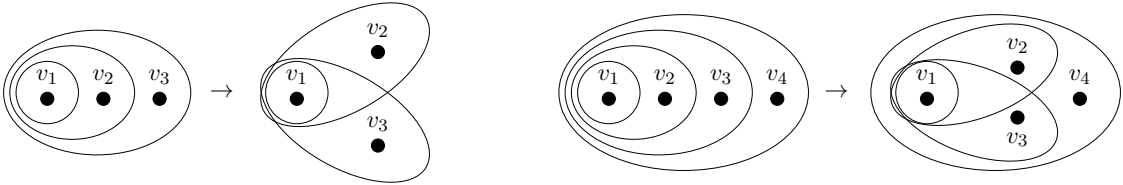

    \centering
    \includestandalone[width=0.9\textwidth]{figures/height-red}
    \caption{Height $2$ reduction in \cref{lem: height2} (left) and height $3$ reduction in \cref{lem: height3} (right)}
    \label{fig:height-2-red}
\end{figure}

\begin{lemma}\label{lem: height3}
    Let $i\in \{3,4\}$. Let $\tau$ be a finite topology on $X$. Then there exists a finite topology $\tau'$ on $X$ such that $G_i(\tau)\cong G_i(\tau')$ and the height of $\tau'$ is at most $3$.
\end{lemma}

\begin{proof}
Suppose the theorem is false and let $\tau$ be a finite topology on $X$ which is a minimum counterexample with respect to the height $h$ of $\tau$ and subject to that, the number of distinct $4$-chains in $\tau$. If the height of $\tau$ is at most $3$ then we are done so suppose the height of $\tau$ is at least $4$. Let $(v_1,v_2,v_3,v_4)$ be a $4$-chain in $\tau$.

We now form $\tau'$ from $\tau$ by removing $v_2$ from every set which contains $v_3$ but not $v_4$. Thus, $(v_1,v_2,v_3,v_4)$ is no longer a $4$-chain in $\tau'$. Since the height of $\tau'$ is at most the height of $\tau$, what remains to show is that $G_i(\tau)\cong G_i(\tau')$, contradicting the choice of $\tau$.  Certainly if $u$ and $v$ are $T_i$-separated in $\tau$, they are still $T_i$-separated in $\tau'$, thus, what remains to show is that if $u$ and $v$ are $T_i$-adjacent in $\tau$, they remain $T_i$-adjacent in $\tau'$. Let $u$ and $v$ be elements of $X$ which are $T_i$-adjacent. Without loss of generality, let
\[ x\in \begin{cases}
    m_{\tau}(\ell(u))\cap m_{\tau}(v
) &\text{ if $i = 3$}\\
m_{\tau}(\ell(u))\cap m_{\tau}(\ell(v)
) &\text{ if $i = 4$}.\\
\end{cases}\]

The only way for $u$ and $v$ to be $T_i$-separated in $\tau'$ is if $x = v_2$ or if $v_2\in \ell_{\tau}(u)$ and $x\in m_{\tau}(v_2)$. In the first case, since $m_{\tau'}(v_1)\subseteq m_{\tau'}(v_2)$, $v_1\in m_{\tau'}(\ell(u))\cap m_{\tau'}(v)$ (and likewise for $i=4$). In the second case, since $m_{\tau'}(v_2)\subseteq m_{\tau'}(v_4)$, $v_4\in \ell_{\tau'}(u)$ and since $x\in m_{\tau}(v_2)$, $x\in m_{\tau}(v_4)$. Thus, $x\in m_{\tau'}(\ell(u))\cap m_{\tau'}(v)$ (and likewise for $i = 4$).
\end{proof}

\section{Proof of \cref{thm:poset-characterization}}\label{sec: comp proof}
Before proving \cref{thm:poset-characterization}, we prove two useful lemmas.

\begin{lemma}\label{lem:sub-topology}
Let $\tau$ be a finite topology on $X$ and let $i\in \{0,1,2,3,4\}$. Let $\tau_S$ be the induced subspace topology of $\tau$ generated by $S\subseteq X$. If $u,v\in S$ are $T_i$-adjacent in $\tau_S$, they are $T_i$-adjacent in $\tau$.
\end{lemma}
\begin{proof}
Suppose $u$ and $v$ are $T_i$-adjacent in $\tau_
S$. First, let $i = 0$ and suppose for contradiction that $m_\tau(u) \not= m_\tau(v)$. Then there exists $a,b \in X \setminus S$ with $a \in m_\tau(u) \setminus m_\tau(v)$ and $b \in m_\tau(v) \setminus m_\tau(u)$. But then $u,v \in m_\tau(u) \cap m_\tau(v)$ and $a,b \not\in m_\tau(u) \cap m_\tau(v)$, contradicting the assumption that $m_\tau(u)$ and $m_\tau(v)$ are minimal sets. 

Similarly, suppose $i =1$ and without loss of generality let $m_{\tau_S}(u)\subseteq m_{\tau_S}(v)$. Suppose for contradiction that $m_{\tau}(u)
\not \subseteq m_{\tau}(v)$. Then there exists some $w\in m_{\tau}(u)\setminus m_{\tau}(v).$ But then $m_{\tau}(u)\cap m_{\tau}(v)$ contains $u$ but not $w$ which contradicts the minimality of $m_{\tau}(u)$.

If $i\geq 2$, then elements being $T_i$-adjacent in $\tau_S$ corresponds to two sets $A$ and $B$ in
$\tau_S$ intersecting. By adding back in the elements $X\setminus S$, open sets which intersect in $\tau_S$ still intersect in $\tau$. This implies that $u$ and $v$ are $T_i$-adjacent in $\tau$ too.
\end{proof}

\begin{lemma}\label{lem: sep to dist}
Let $\tau$ be a finite topology on X and let $H = G_1(\tau)$. If $u,v\in X$ are $T_i$-adjacent for $i\in \{1,2,3,4\}$, then $\text{dist}_H(u,v)\leq i$.
\end{lemma}

\begin{proof}
The claim is trivial when $i = 1$. Suppose $i = 2$. If $u$ and $v$ are $T_2$-adjacent, there exists $z\in m_{\tau}(u)\cap m_{\tau}(v)$. By \cref{obs: topology}(\ref{obs:intersection}) this means $z\in N_H(u)\cap N_H(v)$ and hence $\text{dist}_H(u,v)\leq 2$. Next, suppose $i =3$ ($3'$ or $3''$). If $u$ and $v$ are $T_3$-adjacent, then by \cref{obs: topology}(\ref{obs: ell must exist}) we may assume without loss of generality there exists some $z\in \ell_{\tau}(u)$ such that $z$ and $v$ are $T_2$-adjacent. By the previous case, $\text{dist}_H(z,v)\leq 2$. By \cref{obs: topology}(\ref{obs: ell intersection}), $z\in N_H(u)$ and thus, $\text{dist}_H(u,v)\leq 3$. Finally, suppose $i=4$. If $u$ and $v$ are $T_4$-adjacent, then by \cref{obs: topology}(\ref{obs: ell must exist}) there exists $z\in \ell_{\tau}(u)$ and $z'\in \ell_{\tau}(v)$ such $z$ and $z'$ are $T_2$-adjacent. By the case when $i=2$, $\text{dist}_H(z,z')\leq 2$. By \cref{obs: topology}(\ref{obs: ell intersection}), $z\in N_H(u)$ and $z'\in N_H(v)$ and thus, $\text{dist}_H(u,v)\leq 4$.
\end{proof}

\begin{remark}\label{rem:proof of sep to dist}
We note a couple of things with regard to \cref{lem: sep to dist}. The first is that the contrapositive of \cref{lem: sep to dist} will be applied frequently, which states that if $\text{dist}_H(u,v)>i$, then $u,v$ are $T_i$-separated. Secondly, in the proof of \cref{lem: sep to dist} we see that if $u$ and $v$ are $T_2$-adjacent and $\text{dist}_H(u,v)=2$ via the path $uzv$, then $m_{\tau}(z)\subseteq m_{\tau}(u)$ and $m_{\tau}(z)\subseteq (v)$. Similarly, if $u$ and $v$ are $T_4$-adjacent, and $\text{dist}_H(u,v)=4$ via the path $uzwz'v$, then $m_{\tau}(u)\subseteq m_{\tau}(z)$, $m_{\tau}(v)\subseteq m_{\tau}(z')$, $m_{\tau}(w)\subseteq m_{\tau}(z)$, and $m_{\tau}(w)\subseteq m_{\tau}(z')$.
\end{remark}

We now restate \cref{thm:poset-characterization} for convenience.

\begin{theorem*}[\ref{thm:poset-characterization}]
The following holds.
\begin{enumerate}
    \item A graph $G\in \G_0$ if and only if $G$ is a disjoint union of cliques.
    \item A graph $G\in \G_1$ if and only if $G$ is a comparability graph.
    \item A graph $G\in \G_2$ if and only if $G$ is an upper bound graph.
    \item A graph $G\in \Gthreeand$ if and only if $G$ is a half-closed upper bound graph.
    \item A graph $G\in \Gthreeor$ if and only if $G$ is a fully-closed upper bound graph.
    \item A graph $G\in \G_4$ if and only if $G$ is an extended-closed upper bound graph.
\end{enumerate}
\end{theorem*}

\begin{proof}
We have already shown (1) in \cref{prop:G0}. We now show (2). First, let $G\in \G_1$ and let $\tau$ be a finite topology on $X$ such that $G_i(\tau)\cong G$. By \cref{lem:m(x)=m(y)} we can assume that $m_{\tau}(x)\neq m_{\tau}(y)$ and thus each element of $X$ has a unique minimal set. Thus, $G$ is the containment graph using these minimal sets as $\Sigma$ and by \cref{thm: cont=comp}, $G$ is a comparability graph. Similarly, if $G$ is a comparability graph by \cref{thm: cont=comp}, it is also a containment graph. By mapping each element of $X$ to a minimal set with respect to the containment graph we obtain a finite topology $\tau$ such that $G\cong G_1(\tau)$. That is, $G\in \G_1$.

Now that the complete characterization is established for $\G_1$ and since $E(G_1(\tau))\subseteq E(G_i(\tau))$ for all $i\geq 1$, we show that for (3)-(6), $G_i(\tau)$ is in the hypothesized poset graph class where the poset we are referring to is the one in which $G_1(\tau)$ is the comparability graph of. Again by \cref{lem:m(x)=m(y)}, we may assume $m_{\tau}(x)\neq m_{\tau}(y)$ for all $x\neq y$. We let $O$ be the transitive orientation of $G_1(\tau)$ obtained by orienting $u\to v$ if and only if $m_{\tau}(v)\subset m_{\tau}(u)$.

For (3)-(6), we consider $u,v\in V(H)$ with $\text{dist}_H(u,v)= k$ for $k\geq 2$. We will always let $P = ux_1...x_{k-1}v$ be a shortest $u,v$-path in $H$. Since $P$ is an induced path by definition, there is either one or two ways to transitively orient $P$ depending on the parity of $k$ up to symmetry. If $k$ is even, there are two orientations and if $k$ is odd, there is only one. These orientations are the ones in which every vertex has either in-degree zero or out-degree zero alternating as depicted in \cref{fig:comp-paths}.

We let $\tau_P$ be the subspace topology of $\tau$ whose elements are $V(P)$ and whose sets are the ones obtained from $\tau$ by removing the elements which are not $V(P)$. By \cref{lem:sub-topology}, to prove that $u$ and $v$ are $T_i$-adjacent in $\tau$, it suffices to show that $u$ and $v$ are $T_i$-adjacent in $\tau_P$. Since $\tau_P$ follows a specific structure (shown in \cref{fig:comp-paths}), we can explicitly describe $m_{\tau_P}$ for all of $V(P)$. Using the labeling (i.)-(v.) as in \cref{fig:comp-paths}, we have the following information of the minimal sets for $V(P)$. 

\textbf{Minimal sets when $k=2$.} In (i.) $u\leftarrow x_1\to v$: $m_{\tau_P}(u) = \{u\}$, $m_{\tau_P}(x_1) = \{u,x_1,v\}$, and $m_{\tau_P}(v) = \{v\}$. Thus, $m_{\tau_P}(u)\cap m_{\tau_P}(v)\neq \emptyset$, and $u$ and $v$ are $T_2$-separated in $\tau_P$. Since $m_{\tau_P}(\ell(u))=m_{\tau_P}(\ell(v)) = \{u,x_1,v\}$, we have $m_{\tau_P}(\ell(u))\cap m_{\tau_P}(v)  \neq \emptyset$ and vice versa, implying that $u$ and $v$ are $\Tthreeand$-adjacent in $\tau_P$. In (ii.) $u\to x_1 \leftarrow v$: $m_{\tau_P}(u) = \{u,x_1\}$, $m_{\tau_P}(x_1) = \{x_1\}$, and $m_{\tau_P}(v) = \{x_1,v\}$. Hence, $u$ and $v$ are $T_2$-adjacent in $\tau_P$.

\textbf{Minimal sets when $k=3$.} In the only transitive orientation up to symmetry, (ii.) $u \to x_1 \leftarrow x_2 \to v$, we have the following minimal sets. $m_{\tau_P}(u)=\{u,x_1\}$, $m_{\tau_P}(x_1) = \{x_1\}$, $m_{\tau_P}(x_2) = \{x_1,x_2,v\}$, and $m_{\tau_P}(v)=\{v\}$. Since $m_{\tau_P}(\ell(v))=\{x_1,x_2,v\}$, we have $m_{\tau_P}(u)\cap m_{\tau_P}(\ell(v))\neq \emptyset$ implying that $u$ and $v$ are $\Tthreeor$-adjacent in $\tau_P$. Note that if there was also a $v,u$-path (oriented like (ii)), then $u$ and $v$ would be $\Tthreeand$-adjacent in $\tau_P$.

\textbf{Minimal sets when $k=4$.} In (iv.) $u\leftarrow x_1 \to x_2 \leftarrow x_3 \to v$: the minimal sets are as follows. $m_{\tau_P}(u) = \{u\}$, $m_{\tau_P}(x_1) = \{u,x_1,x_2\}$, $m_{\tau_P}(x_2)=\{x_2\}$, $m_{\tau_P}(x_3) = \{x_2,x_3,v\}$, and $m_{\tau_P}(v) = \{v\}$. Since $m_{\tau_P}(\ell(u)) = \{u,x_1,x_2\}$, and $m_{\tau_P}(\ell(v)) = \{x_2,x_3,v\}$ we see that $m_{\tau_P}(\ell(u))\cap m_{\tau}(v)=\emptyset$, $m_{\tau_P}(u)\cap m_{\tau_P}(\ell(v))=\emptyset$ but $m_{\tau_P}(\ell(u)\cap m_{\tau_P}(\ell(v))\neq \emptyset$ implying that in $\tau_P$, $u$ and $v$ are $\Tthreeor$-separated but $T_4$-adjacent. In (v.) $u\to x_1 \leftarrow x_2 \to x_3 \leftarrow v$: the minimal sets are as follows. $m_{\tau_P}(u) = \{u,x_1\}$, $m_{\tau_P}(x_1) = \{x_1\}$,
$m_{\tau_P}(x_2) = \{x_1,x_2,x_3\}$,
$m_{\tau_P}(x_3) = \{x_3\}$, and
$m_{\tau_P}(v) = \{x_3,v\}$. Since $m_{\tau_P}(\ell(u)) = \{u,x_1\}$ and $m_{\tau_P}(\ell(v))=\{x_3,v\}$, $m_{\tau_P}(\ell(u))\cap m_{\tau_P}(\ell(v))=\emptyset$ implying that $u$ and $v$ are $T_4$-separated.

We are now ready to prove (3)-(6). By the contrapositive of \cref{lem: sep to dist}, to obtain $E(G_i(\tau))$ we need only to consider the case when $\text{dist}_H(u,v)=k$ for $k\leq i$. First, consider $i=2$. If $u$ and $v$ are along a path of type (i.) implying they are $T_2$-separated in $\tau_P$ but are $T_2$-adjacent in $\tau$, then by \cref{rem:proof of sep to dist} there exists a path of type (ii.) from $u$ to $v$ also. This means $E(G_2(\tau))$ can be obtained from $E(H)$ by adding edges between $u$ and $v$ if and only if there exists $z\in N^+_O[u]\cap N^+_O[v]$. Since $O$ is the transitive orientation corresponding to a poset $P=(V(H),\leq)$, where $H$ is the comparability graph of $P$, then there exists vertices in $G$ with $u\to w$ and $v\to z$ if and only if $u\leq w$ and $v\leq z$ and hence $u$ and $v$ have a common upper bound. This implies that $G_2(\tau)$ is the upper bound graph of $P$. This proves (3).

Next consider $i = 3'$. If $\text{dist}_H(u,v)\leq 2$, then we see that $u$ and $v$ are $\Tthreeand$-adjacent in $\tau$. If $\text{dist}_H(u,v) = 3$ and $u,v$ are $\Tthreeand$-adjacent, then $H$ must contain a $u,v$-path and a $v,u$-path of type (iii). This means $E(G_{3'}(\tau))$ can be obtained from $E(H)$ by adding edges between $u$ and $v$ if and only if there exists some $a,b,\alpha,\beta\in V(H)$ and path in the transitive orientation $O$ of $H$ such that $u\leftarrow a \to b \leftarrow v$ and $u \to \alpha, \leftarrow \beta \to v$. Since $O$ is the transitive orientation corresponding to the poset $P = (V(H), \leq)$, we have $a\leq u$, $a\leq b, v\leq b$, and $\alpha\leq v$, $\alpha \leq \beta$, $u\leq \beta$. Thus, $u$ and $v$ satisfy the definition of adjacency in the half-closed upper bound graph of $P$. This proves (4).

Similarly, since $u,v$ are $\Tthreeor$-adjacent for all paths of length at least $3$, $E(G_3(\tau))$ can be obtained from $E(H)$ by adding edges between $u$ and $v$ if $\text{dist}_H(u,v)\leq 3$, i.e. $E(G_3(\tau)) = E(H^3)$. Since $O$ is the transitive orientation corresponding to a poset $P=(V(H),\leq)$, where $H$ is the comparability graph of $P$, then there exists vertices of $G$ with either $u\leftarrow a \to b \leftarrow v$ or $u \to \alpha, \leftarrow \beta \to v$  if and only if either $a \leq u$, $a\leq b$, and $v\leq b$, or $\alpha \leq v$, $\alpha\leq \beta$, and $u\leq \beta$, Thus, $u$ and $v$ satisfy the definition of adjacency in the fully-closed upper bound graph of $P$, which proves (5).

Finally, consider $i=4$. If $u,v$ are endpoints of a path of type (v.), implying they are $T_4$-separated in $\tau_P$, but $u$ and $v$ are $T_4$-adjacent in $\tau$, then by \cref{rem:proof of sep to dist} there exists a $u,v$ path of type (iv.) too. Thus, $E(G_4(\tau))$ can be obtained from $E(H)$ by adding edges between $u$ and $v$ if and only if there exists $a\in N^-_O[u]$ and $b\in N^-_O[v]$ such that $N^+_O[a]\cap N^+_O[b]\neq \emptyset$. Let $c\in N^+_O[a]\cap N^+_O[b]$. Since $O$ is the transitive orientation corresponding to a poset $P=(V(H),\leq)$, where $H$ is the comparability graph of $P$, then there exists vertices of $G$ with $u \leftarrow a\to c$ and $v \leftarrow b\to c$ if and only if $a\leq u$, $b\leq v$, $a\leq c$, and $b\leq c$ and hence $u$ and $v$ satisfy the definition of adjacency in the extended-closed upper bound graph of $P$. This proves (6) and completes the proof.
\end{proof}

\begin{figure}[h!]
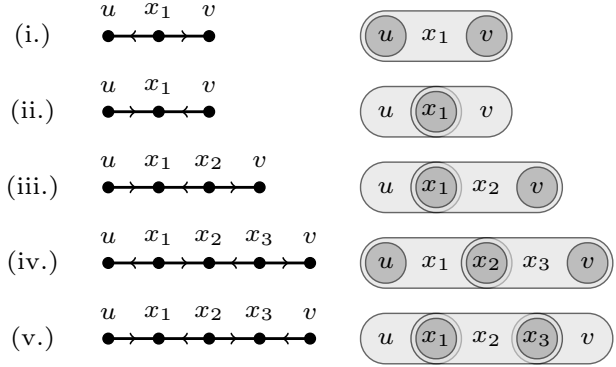

    \centering
    \includestandalone[width=0.5\textwidth]{figures/comp-paths}
    \caption{Paths considered in \cref{thm:poset-characterization} with transitive orientation (left) and visual interpretation of the minimal base of associated sub-topology $\tau_P$ (right)}
    \label{fig:comp-paths}
\end{figure}

\section{A closer look into $\G_2$}\label{sec:More with T2}

Unlike $\G_0$ and $\G_1$, the graph class $\G_2$ is not hereditary. A hereditary graph class is one which is closed under taking induced subgraphs. To see that $\G_2$ is not hereditary, consider any graph $H$. We define the \textit{incidence topology on $V(H)\cup E(H)$}, denoted $\text{Inc}(H)$, whose minimal base is $M(\text{Inc}(H)) = \bigcup_{e\in E(H)} \{e\} \cup \bigcup_{v\in V(H)} \{v,\{e\in E(G) : e \text{ is incident to $v$}\}\}$. Note that the graph $G_2(\text{Inc}(H))$ can be obtained from $G$ by duplicating each edge of $E(H)$, and then subdividing one of the parallel edges from each pair. Hence, the edges that are not subdivided induce a copy of $H$.

In \cref{thm:poset-characterization} we proved that graphs in $\G_2$ are equivalent to upper bound graphs. Thus, by \cref{thm:upperbound}, graphs in $\G_2$ are also equivalent to edge-simplicial graphs. In order to understand the connections between the finite topologies and the clique cover guaranteed in edge-simplicial graphs, we present a proof showing directly that graphs in $\G_2$ have the desired clique cover for edge-simplicial graphs. This will also provide us a template for characterizing graphs in $\Gthreeand$ in \cref{sec:G3and}.

\begin{theorem}\label{thm:G2-characterization}
A graph $G\in \G_2$ if and only if $G$ contains a clique cover with the property that every clique contains a simplicial vertex.
\end{theorem}

\begin{proof}
We first show that if $G\in \G_2$ then it has the desired clique cover. Since $G\in \G_2$, there exists a topology $\tau$ on $V(G)$ such that $G\cong G_2(\tau)$. By \cref{lem:m(x)=m(y)} and \cref{lem: height2}, we may assume $\tau$ has height $2$ and for any $x,y\in V(G)$, $m_{\tau}(x)\neq m_{\tau}(y)$. Since $\tau$ has height $2$, we let $L_1\cup L_2$ be a partition of $V(G)$, into two layers, such that $m_{\tau}(v) = \{v\}$ for all $v\in L_1$ and $m_{\tau}(u)\setminus \{u\} \subseteq L_1$ for all $u\in L_2$.

\begin{claim}
For any $v\in L_1$, $N[v]$ induces a clique.
\end{claim}
\begin{proofc}
Let $v\in L_1$. By definition $m_{\tau}(x)= \{x\}$ for any $x\in L_1$ and hence $N(v)\subseteq L_2$. Let $u\in N(v)\cap L_2$. Then the only way for $u$ to be in $N(v)$ is if $v\in m_{\tau}(u)$. Since this holds for any $u \in  N(v)$, it must be the case that $v\in \bigcap_{u\in N(v)} m_{\tau}(u)$. Thus, $N[v]$ induces a clique.
\end{proofc}

Now consider $\C = N[v_1],N[v_2],...,N[v_t]$ where $L_1 = \{v_1,...,v_t\}$. We claim that $\C$ is a clique cover and moreover by the previous claim each $v_i$ is a simplicial vertex. Since $m_{\tau}(v) = \{v\}$ for all $v\in L_1$, $L_1$ forms an independent set and hence we must only check that edges between $L_1$ and $L_2$ and edges within $L_2$ are covered. Clearly edges between $L_1$ and $L_2$ are covered, as $\C$ consists of closed neighborhoods of vertices in $L_1$. Next, let $u,w\in L_2$ with $uw\in E(G)$. If $uw\in E(G)$, then there exists some $v\in m_{\tau}(u)\cap m_{\tau}(w)\cap L_1$. This means $uw$ is covered by $N[v]$. Thus, $\C$ is a clique cover such that every clique contains a simplicial vertex.

Now suppose $G$ is a graph with clique cover $\C = C_1,...,C_t$ such that every clique has a simplicial vertex, say $v_i$ in $C_i$. Let $L_1 = \{v_1,...,v_t\}$ and let $L_2 = V(G)\setminus L_1$. We can construct the minimal sets in a topology $\tau$ by letting $m_{\tau}(v) = \{v\}$ for all $v\in L_1$ and $m_{\tau}(u) = \{u\} \cup \{v\in L_1\cap N(u)\}$. It can easily be seen that $\tau$ is a topology such that $G_2(\tau)\cong G$ by the same logic as the forward direction.
\end{proof}

\begin{figure}[h!]
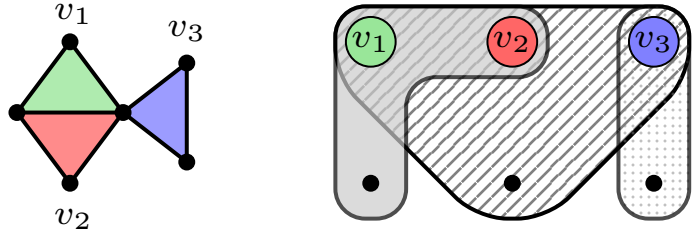

    \centering
    \includestandalone[width=0.55\textwidth]{figures/construction}
    \caption{An example of the backwards direction of \cref{thm:G2-characterization}. On the left: a graph $G$ with clique cover $\C$ where every clique contains a simplicial vertex and on the right: the minimal base of the topology $\tau$ such that $G \cong G_2(\tau)$.}
    \label{fig:core-construction}
\end{figure}

By \cref{lem: height2} and as seen in the backwards direction of \cref{thm:G2-characterization}, we can partition the vertices of $G\in \G_2$ into two layers. Because of this, we can interpret layer one ($L_1$) as a vertex set and layer two ($L_2$) as the edges of some hypergraph. Now, the minimal sets of $\tau$ are the vertices themselves and for each edge, it includes the edge along with a set of vertices it is incident to. By translating this into the language of hypergraphs, this means precisely that $\G_2$ is the \textit{middle graph} of a hypergraph. The middle graph of a hypergraph $H = (V,E)$ is the graph whose vertex set is $V\cup E$ and the edges are between edge-to-edge and edge-to-vertex adjacencies. Borowiecki \cite{MIDDLE1} proved that a graph is edge-simplicial if and only if it is the middle graph of some hypergraph. Thus, we readily see that $\G_2$ is precisely the class of graphs which are the middle graphs of hypergraphs.

For any graph $G$, the \textit{clique cover number} of $G$ is the minimum number of cliques in any clique cover of $G$. In 1972, Karp proved that finding the clique cover number of an arbitrary graph is NP-complete, one of the first problems to be classified as NP-complete \cite{KARP}. However, certain classes of graphs have polynomial time algorithms to determine the clique cover number. One such class of graphs are perfect graphs, which in 1981 were proven by Gr\"otschel, Lov\'asz, and Schrijver to have a polynomial time algorithm to find their clique cover number \cite{Grotschel}. Since comparability graphs were shown to be perfect graphs by Berge \cite{BERGE1}, the clique cover number of graphs in $\G_1$ can be found in polynomial time. It is easy to see that edge-simplicial graphs (graphs in $\G_2$) also have a polynomial time algorithm to determine their clique cover number, see \cite{UB-SURVEY}. It is worth noting that since $\G_2$ is not hereditary, there are graphs in $\G_2$ that are not perfect.

\subsection{Clique cover graphs of $\G_2$}\label{sec:clique covers}
After seeing \cref{thm:G2-characterization}, it is natural to wonder what the clique covers characterizing graphs in $\G_2$ look like. To this end we analyze the \textit{clique cover graphs} of graphs in $\G_2$. Given a graph $G$ and a clique cover $\C$, we define $H_G(\C)$ as the graph whose vertices are the cliques of $\C$, and two cliques are adjacent if and only if they share at least one vertex in $G$. Since in $\G_2$ cliques can have arbitrary size, we let $\G_2^k$ denote the graph $G\in \G_2$ with $\omega(G)\leq k$. In 1943 in \cite{KRAUSZ}, J. Krausz  proved that a graph $G$ is a line graph of some multigraph if and only if $G$ has a clique cover such that every vertex is in at most $2$ cliques (see \cite{HARARY} Theorem 8.4 pg 72). Berge then generalized this to line graphs of hypergraphs in \cite{BERGE} showing that a graph $G$ is a line graph of a hypergraph, with rank at most $k$, if and only if $G$ has a clique cover such that every vertex is in at most $k$ cliques. Using the characterization from \cref{thm:G2-characterization}, it will be fairly easy to see that in fact, clique cover graphs of $\mathcal{G}_2^k$ have exactly the required clique covers mentioned above. 

\begin{theorem}\label{thm: ccgraph}
    Let $k\geq 3$. $H$ is the clique cover graph of some $G\in \G_2^k$ if and only if $H$ is a line graph of a rank $k-1$ hypergraph.
\end{theorem}

\begin{proof}
Let $\tau$ be a finite topology on $X$ and let $G\in \G_2^k$ with clique cover $\C = \{C_1,...,C_t\}$ such that each $C_i$ contains a simplicial vertex. The existence of $\C$ is guaranteed by \cref{thm:G2-characterization}. Let $H = H_G(\C)$. We prove that $H$ has a clique cover $\C' = \{C'_1,...,C'_s\}$ such that each vertex $C_i\in V(H)$ is in at most $k-1$ cliques. Then by Berge's characterization for line graphs of hypergraphs we obtain our desired result. Afterwards, we will give the explicit construction to obtain a $k-1$ uniform hypergraph $K$ such that $L(K) \cong H$.

We define $\C'$ as follows. First let $x_1,...,x_s$ be the vertices of $G$ which are \textit{not} simplicial vertices. Now let $C'_i = \{C_j\in \C:x_i\in C_j\}$. We first show that $\C'$ is in fact, a clique cover of $H$. Let $C_iC_j\in E(H)$. By definition, in $G$, $V(C_i)\cap V(C_j)\neq \emptyset$, let $x_{\ell}\in V(C_i)\cap V(C_j)$. Then $x_\ell$ is not a simplicial vertex and thus, $C_iC_j$ is covered by clique $Y'_{\ell}$. Next, suppose for contradiction that $C_i$ is contained in at least $k$ cliques say $C'_1,...,C'_k$. This means though that there are at least $k$ non-simplicial vertices in $C_i$. Since $|V(C_i)|\leq k$ by definition of $\G_2^k$, this means $C_i$ does not contain any simplicial vertices. This contradicts the property of $\C$ and completes the proof.
\end{proof}

\begin{figure}[h!]
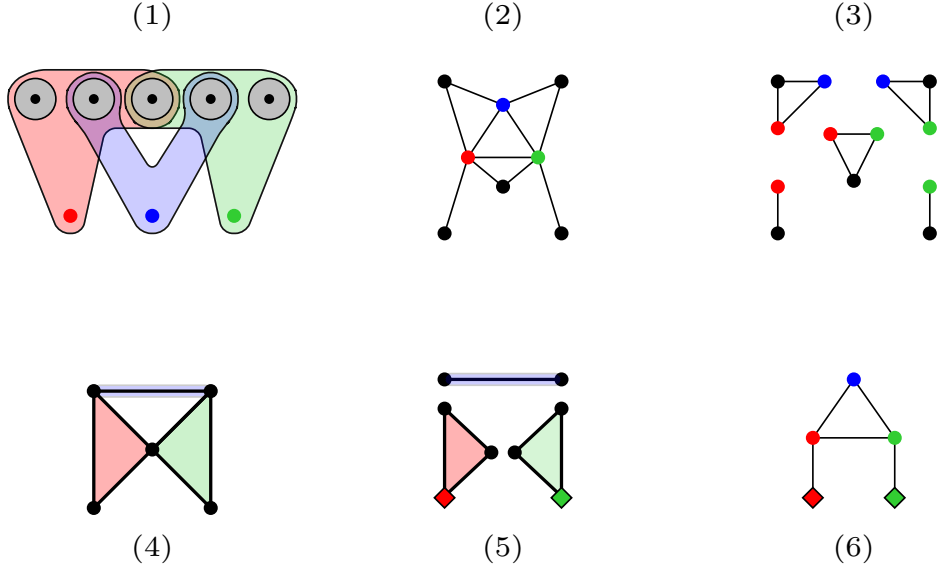

    \centering
    \includestandalone[width=0.75\textwidth]{figures/the-key}
    \caption{(1) The minimal base of a topology $\tau$, (2) $G = G_2(\tau)$, (3) The clique cover $\C$ satisfying \cref{thm:G2-characterization}, (4) $H = H_G(\C)$, (5) The clique cover of $H$ satisfying Berge's characterization of line graphs of hypergraphs \cite{BERGE}, (6) the graph $K$ with $L(K) = H$.}
    \label{fig: linegraph ex}
\end{figure}

Using the notation in the proof of \cref{thm: ccgraph}, we now explicitly give the construction to obtain, from a finite topology $\tau$, the rank $k-1$ hypergraph $K$ such that $L(K) \cong H$ where $H = H_G(\C)$ and $G\in \G^k_2$ with clique cover $\C = \{C_1,...,C_t\}$ with the property that every clique contains a simplicial vertex, guaranteed by \cref{thm:G2-characterization}. We let $\C' = \{C'_1,...,C'_s\}$ be the clique cover of $H$, where $C'_i = \{C_j\in \C:x_i\in C_j\}$ and $V(G)\setminus \{v\in V(G): \text{$v$ is a simplicial vertex}\} =x_1,...,x_s$ are the non-simplicial vertices of $G$. See \cref{fig: linegraph ex} for an example. We now define our rank $k-1$ hypergraph $K$. \begin{itemize}
    \item $V(K) = V(G)\setminus \{v\in V(G): \text{$v$ is a simplicial vertex}\}$.
    \item $E(K) = \{E_i : i\in [t]\}$ where for each $i\in [t]$, $E_i = \{v\in V(K) : v\in C_i\}$.
    \end{itemize}

Since there is one edge of $K$ for every clique in $\C$, $V(L(K)) = V(H)$ and by definition of taking a line graph, $E(L(K))=E(H)$ since two cliques are adjacent in $H$ precisely when they intersect. As mentioned in the proof of \cref{thm: ccgraph}, $K$ has rank $k-1$ since $\omega(G)\leq k$ and every clique of $\C$ has at least one simplicial vertex by \cref{thm:G2-characterization}. Note by adding dummy vertices (indicated as diamonds in \cref{fig: linegraph ex}) one can make $K$ a $k-1$-uniform hypergraph (a loopless graph in the case when $k= 3$).

\subsection{Critical graphs in $\G_2$}\label{sec:edge-critical}

In any graph class $\G$, it is natural to wonder which graphs are \textit{critical}. In this context, by critical we mean that $G\in \G$ but $G - S\notin \G$ for any $S\subseteq E(G)$. Since removing all the edges of graph in $\G_2$ yields another graph which is trivially in $\G_2$, we restrict ourselves to removing non-cut sets. Formally we say a graph $G$ is \textit{$\G_2$-critical} if in $G\in \G_2$ and for any set $S\subseteq E(G)$, $G-S$ is either disconnected or $G-S\notin \G_2$. Let $\G_2^*$ denote the set of $\G_2$-critical graphs.

Trees exhibit this ``critical" behavior among connected graphs in the sense that the removal of any edge will no longer yield a connected graph. Moreover, trees are graphs which can be built iteratively by adding a vertex and connecting it to an already existing vertex. In a similar manner, we see that $\G_2$-critical graphs also have tree-like behavior and can be built in an iterative process.

\begin{definition}
Given a graph $G$, \emph{anchoring a star to  $G$} is the operation of adding a $K_{1,t}$, for $t\geq 1$, say with vertex $x$ in the part of size $1$, and joining $x$ to a maximal clique of $G$.
\end{definition}

We restate \cref{thm:critical-new} for convenience.

\begin{theorem*}[\ref{thm:critical-new}]
$G\in \G_2^*$ if and only if $G\cong K_2$ or $G$ is connected and can be obtained by iteratively anchoring stars to $K_1$.
\end{theorem*}

\begin{proof}
It is easy to see that $K_2\in\G_2^*$ so now we assume $G\not \cong K_2$. We prove both directions of this via a series of claims. We first show that such hypothesized graphs are $\G_2$-critical.

\begin{claim}
If $G$ can be obtained by iteratively anchoring stars to a $K_1$, then $G\in \G_2^*$.
\end{claim}

\begin{proofc}
  To show this, we first show that $G\in \G_2$ and then show that the removal of any edge set $S$ either disconnects $G$ or $G-S\notin \G_2$. To do this, we prove the stronger statement that if $H$ satisfies the following two conditions, (R1) and (R2), then anchoring a star to $H$ produces a graph which still satisfies conditions (R1) and (R2). Then since $K_1$ trivially satisfies (R1) and (R2), we inductively prove our claim. The conditions are as follows.

  \begin{enumerate}
      \item[(R1)] $H\in \G_2^*$ and $H\not \cong K_2$.
      \item[(R2)] Every maximal clique in $H$ is in the clique cover of $H$, from \cref{thm:G2-characterization} in which every clique contains a simplicial vertex.
  \end{enumerate}

Consider a graph $H$ satisfying conditions (R1) and (R2). By \cref{thm:G2-characterization}, there exists a clique cover of $H$, with maximal cliques, say $\C_H = \{C_1,...,C_k\}$ such that every clique has a simplicial vertex. Let $H'$ be the graph obtained by anchoring a star, $K_{1,t}$ where $t\geq 1$, to $H$. Let the vertices of the $K_{1,t}$ be $x,y_1,.,,,y_t$, where $x$ is the universal vertex. Since $H$ satisfies condition (R2), without loss of generality let $C_1$ be the maximal clique of $H$ in which we join $x$ to. Let $v_1$ be a simplicial vertex of $C_1$. Note that $v_1$ is still a simplicial vertex of $H'$ and thus, $\C_{H'} = \{C_1\cup x\} \cup \{C_i : 2\leq i \leq k\} \cup \bigcup_{i=1}^t\{xy_i\}$ is clique cover of $H'$ in which every clique has a simplicial vertex (note $y_i$ is a simplicial vertex for all $i\in [t]$). Thus, $H'\in \G_2$. Moreover, every maximal clique of $H'$ is in $\C_{H'}$ since $x$ is joined to a maximal clique. Hence, $H'$ satisfies conditions (R2).

Next we show that $H'\in \G_2^*$. To see this, first let $S_x$ be a subset of edges incident to $x$ which does not disconnect $H'$. Hence, $xy_i\notin S_x$ for all $i\in [t]$. Let $e = xv \in S_x$ where $v\in V(C_1)$. If $|V(C_1)| = 1$, then $e$ is a cut edge so we may assume $|V(C_1)|\geq 2$. Since $S_x$ is not an edge-cut, there must exist some $u\in V(C_1)\setminus \{v\}$ such that $xu\notin S_x$. Note that $C_1\cup \{x\}$ no longer is a clique in $H'-S_x$ and thus, in any clique covering of $H'-S_x$ with maximal cliques, $xu$ must be covered by a clique which is not $C_1\cup x$. Let $C'$ be a maximal clique in $H'-S_x$ covering $xu$. Note that $V(C')\setminus \{x\}\subseteq V(C_1)\setminus \{v\}$. But then any vertex in $V(C')\setminus \{x\}$ cannot be a simplicial vertex in $H'-S_x$ since it must have both $x$ and $v$ as neighbors. Since $x$ is also not a simplicial vertex in $H'-S_x$, this means $H'-S_x\notin \G_2$. Thus, if $S_x \neq \emptyset$, $H'-S_x\notin \G_2$.

Now suppose for contradiction there is some edge set $S$, such $H'-S\in \G_2$. By the previous paragraph we have that $S\subseteq E(H)$ and since $H\in \G_2^*$, either $H-S$ is disconnected or $H-S\notin \G_2$. If $H-S$ is disconnected but $H'-S$ is not, then since $x$ is joined to a maximal clique, the only way for this to occur is if $x$ is joined to a $K_2$. But this would contradict $H$ satisfying (R1). Now suppose $H-S\notin \G_2$. Then, by \cref{thm:G2-characterization}, there must exist a maximal clique of $H-S$ which has no simplicial vertex. But again since $x$ is joined to a maximal clique, a non-simplicial vertex cannot become a simplicial vertex in $H'$. Thus, $H'-S$ cannot have a clique cover of maximal cliques such that every clique has a simplicial vertex. This contradicts that $H'-
S\in \G_2$.

This proves that in fact, $H'\in \G_2^*$ since otherwise $H$ itself would not be $\G_2$-critical. Thus, $H'$ also satisfies condition (R1) which completes the claim.
\end{proofc}

To prove the other direction, to suffices to show that for any $G\in \G_2$, there exists a connected spanning subgraph $W$ which can be obtained by iteratively anchoring stars to $K_1$. Thus, if $G\not \cong W$, $E(G)\setminus E(W)$ is precisely the desired edge set to demonstrate non-criticality. 

We find such a connected spanning subgraph of $G$ as follows. Let $\C = \{C_1,...,C_k\}$ be the clique cover of $G$ such that each clique of $\C$ has a simplicial vertex, guaranteed by \cref{thm:G2-characterization}. Let $H_G(\C)$ be the clique cover graph of $\C$ and let $T_{\C}$ be a spanning tree of $H_G(\C)$ (see middle of \cref{fig:critical-guide}). Note that since $G$ is connected, $H_G(\C)$ is also connected and thus, $T_{\C}$ exists. We now use $T_\C$ to find a tree $T_G$ in $G$ with the following properties (see left of \cref{fig:critical-guide}).
\begin{enumerate}
    \item[(P1)] For any leaf $C_i$ in $T_\C$, we let $v_i$ be a leaf in $T_G$ where $v_i$ is the corresponding simplicial vertex in clique $C_i$. 
    \item[(P2)] Each non-leaf vertex $C_i$ of $T_\C$ can be mapped to one non-leaf vertex $v_i$ of $\T_G$ such that $v_i$ is in clique $C_i$ and $v_i$ is not a simplicial vertex. 
    \item[(P3)]If $C_iC_j\in E(T_\C)$, then either $v_i = v_j$ or $v_iv_j\in E(T_G)$.
\end{enumerate}

Why can we always find a tree $T_G$ with properties (P1)-(P3)? First note that we can certainly guarantee (P1) since each clique $C_i$ must have a simplicial vertex by \cref{thm:G2-characterization} and we can guarantee (P2) since $G$ is connected and hence if $G$ is not a complete graph, every clique in $\C$ has at least one non-simplicial vertex. Finally, notice that if $C_iC_j\in E(T_\C)$, then it must mean that $V(C_i)\cap V(C_j)\neq \emptyset$ in $G$. This means, after selecting the leaves of $T_G$, we select the non-leafs by considering if $v_i$ has already been selected. Then either $v_i\in V(C_i)\cap V(C_j)$ and $v_i = v_j$ or there exists a $v_j\in V(C_i)\cap V(C_j)$ which by definition must neighbor $v_i$. Thus, we satisfy (P3) and in fact, can construct $T_G$.

We next construct a spanning subgraph $W$ of $G$ from $T_G$ as follows (see right of \cref{fig:critical-guide}). We first add back any edges between non-leaf vertices $u,u'$ if they are adjacent in $G$. Next, for each edge $uv\in E(T_G)$ which is not incident to a leaf, there must be a clique $C_i$ such that $uv\in E(C_i)$, moreover, the simplicial vertex of $C_i$, say $w_i$ in not in $V(T_G)$. For each such edge we first add the simplicial vertex $w_i$ and join it to both $u$ and $v$. Let $W'$ be the graph obtained by adding all such simplicial vertices. Finally, for any vertex $x\in V(G)\setminus V(W')$, $x$ must be in at least one clique $C_i$ where a non-simplicial vertex $v_i$ of $C_i$ is in $V(T_G)$. Join $x$, as a leaf, to one of these such $v_i$'s. To finish the theorem we prove the following claim.

\begin{claim} $W$ is a connected spanning subgraph of $G$ which is either $K_2$ or can be obtained by iteratively anchoring a stars to $K_1$.
\end{claim}

\begin{proofc}
Certainly $W$ is a connected spanning subgraph of $G$. Now let $y_1$ be a leaf of $T_G$. Note that by construction $y_1$ is also a leaf of $W$ and by definition, $y_1$ is a simplicial vertex of some clique, say $C_1$. Let $x$ be the neighbor of $y_1$ in $W$ (and also in $T_G$). If $x$ is also a leaf then $W\cong K_2$, and we trivially have our desired claim. Otherwise, $x$ is not a leaf and hence $x$ must be a non-simplicial vertex representing a clique, say $C_j$ in $T_G$. For every edge $xz\in E(T_G)$, either $z$ is a leaf, or there exists a $w\in N_W(x)\cap N_W(z)$ where $w$ is the simplicial vertex of $C_j$ (see the right of \cref{fig:critical-guide}).  This is the edge incident to $x$ added to form $W'$. For any edge $xz \in E(W)\setminus E(W')$, $z$ is by definition a leaf. Thus, $N_W(x)$ can be partitioned into a set of leaves, say $y_1,...,y_t$, and a subset $C_{W,j}\subseteq C_j$ which is a maximal clique in $W$ containing $x$. Hence, anchoring a star to $W_1:= W - \{x,y_1,...,y_t\}$ obtains $W$. Moreover, in $W_1$, $w$ is a leaf. Since $W$ has at least two leaves to start (namely any two leaves in $T_G$), then $W_1\not \cong K_2$. By iteratively removing anchored stars and repeating this argument, we will eventually get down to a $K_1$.
\end{proofc}

Thus, we see that for each $G\in \G_2$ there exists a spanning subgraph $W$ which is either $K_2$ or can be constructed via successively anchoring a star to $K_1$. Since such graphs are $\G_2$-critical it proves they are the only such graphs which are $\G_2$-critical.
\end{proof}

\begin{figure}[h!]
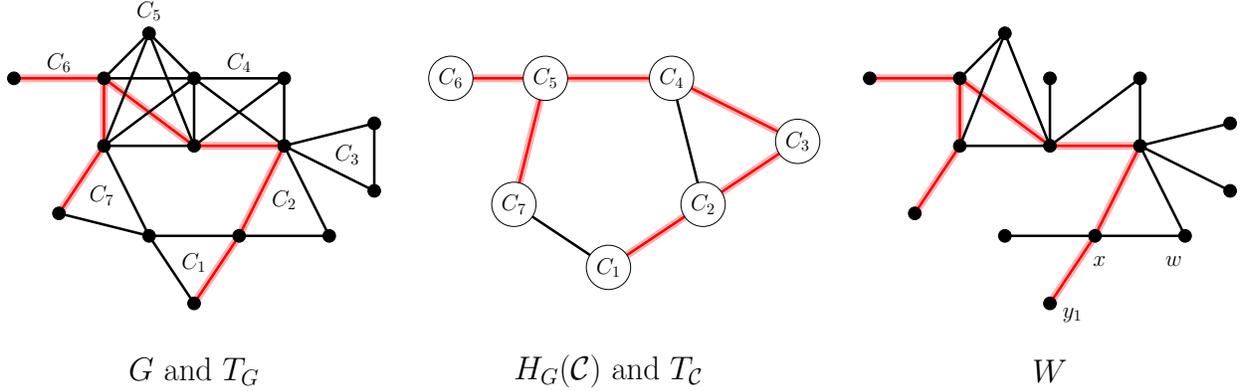

    \centering
    \includestandalone[width=0.99\textwidth]{figures/critical-ex}
    \caption{An example for \cref{thm:critical-new}. In red, $T_G$ (left and right) and $T_\C$ (middle)}
    \label{fig:critical-guide}
\end{figure}

\section{A closer look into $\Gthreeand$}\label{sec:G3and}
The graph class $\Gthreeand$ is the first among the sequence of six graph classes that has not yet been studied in the literature. Before proving \cref{thm:G3and-universe}, we show that no graph in $\Gthreeand$ (and thus, no graph in any successive class) is bipartite.

\begin{theorem}
  Let $G\in \G_i$ for $i\in \{3,4\}$ with $|V(G)|\geq 3$. Then $\chi(G)\geq 3$.
\end{theorem}
\begin{proof}
Suppose for contradiction there exists $G\in \G_i$ for some $i\in \{3,4\}$ such that $G$ is bipartite. Then $G\cong G_i(\tau)$ for some finite topology $\tau$. Since $E(G_2(\tau))\subseteq E(G_i(\tau))$, $G_2(\tau)$ is also bipartite. But then it must be the case that every component of $G_2(\tau)$ is a star as a corollary of \cref{thm:G2-characterization}. 

Suppose there exists a component with at least $3$ vertices, $x,y,z$ where $xz,yz\in E(G_2(\tau))$. Then $m_{\tau}(x)\cap m_{\tau}(y) = \emptyset$, but $m_{\tau}(x)\cap m_{\tau}(z)\neq \emptyset$ and $m_{\tau}(y)\cap m_{\tau}(z)\neq \emptyset$. Let $x'\in m_{\tau}(x)\cap m_{\tau}(z)$ and let $y'\in m_{\tau}(y)\cap m_{\tau}(z)$ (note that $x'$ could be $x$ and $y'$ could be $y$). By \cref{obs: topology}(\ref{obs: ell intersection}), $m_{\tau}(x')\subseteq m_{\tau}(z)$ and $m_{\tau}(y')\subseteq m_{\tau}(z)$. Thus, $z\in \ell_{\tau}(x')\cap\ell_{\tau}(y')$ implying that $x',y'$, and $z$ will form a clique in $G_i(\tau)$ for any $i\in \{3,4\}$. This contradicts $G_i(\tau)$ being bipartite.

Now suppose each component of $G_2(\tau)$ has size at most $2$. We claim this cannot occur since if $G_i(\tau)$ is connected then $G_2(\tau)$ is connected. Suppose not, and let $x,y\in V(G_2(\tau))$ be disconnected vertices. Since they are connected in $G_i(\tau)$ we let $x = x_1,x_2,...,x_k=y$ be an $(x,y)$-path in $G_i(\tau)$. For any consecutive $x_i,x_{i+1}$, since $x_ix_{i+1}\in E(G_i(\tau))$ and by \cref{obs: topology}(\ref{obs: ell must exist}) and \cref{obs: topology}(\ref{obs: ell intersection}), there exists some $v_i\in m_{\tau}(\ell(x_i))\cap m_{\tau}(x_{i+1})$ with $m_{\tau}(v_i)\cap m_{\tau}(x_i)\neq \emptyset$, and $m_{\tau}(v_i)\cap m_{\tau}(x_{i+1})\neq \emptyset$. Hence, $x = x_1v_1x_2v_2...x_{n-1}v_{n-1}x_n = y$ is an $(x,y)$-walk in $G_2(\tau)$ contradicting that $x$ and $y$ are disconnected.
\end{proof}

\begin{definition}\label{def:universe}
A \textit{universe} $U = (s,P,M)$ is a graph with vertex set $V(U) = \{s\} \cup P \cup M$ that satisfies the following conditions.
\begin{enumerate}
    \item $s$ is a universal vertex (the sun).
    \item $U[P]$ induces a clique (the planets).
    \item for all $v\in M$ (the moons), there exists $u\in P$ such that $uv\in E(U)$.
    \item for all $x,y\in M$, $xy\in E(U)$ if and only if there exist $z\in P$ with $xz,yz\in E(U)$.
\end{enumerate}
\end{definition}

A universe $U = (s,P,M)$ generalizes a clique when $M = \emptyset$. Specifically, $U = (s,P,\emptyset)$ is isomorphic to $K_{|P|+1}$. See \cref{fig:universe} for an example of a universe. We are now ready to prove \cref{thm:G3and-universe} we we restate for convenience.

\begin{figure}[h!]
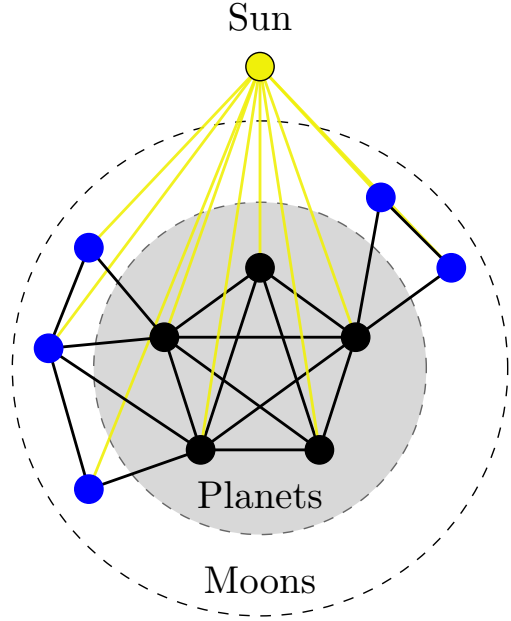

    \centering
    \includestandalone[width=0.4\textwidth]{figures/universe}
    \caption{An example of a universe}
    \label{fig:universe}
\end{figure}

\begin{theorem*}[\ref{thm:G3and-universe}]
$G\in \Gthreeand$ if and only if $G$ has an edge covering with universes $\mathcal{U} = U_1,U_2,...,U_t$ where $U_i = (s_i,P_i,M_i)$, $P = \bigcup_{i=1}^t P_i$, and $S = \{s_1,...,s_t\}$ and the following hold.

\begin{enumerate}
   \item $S\cup P$ partition $V(G)$.
   \item $\forall s_i,s_j \in S$ with $N(s_i) \cap N(s_j) \cap P \not= \emptyset$, then $s_is_j \in E(G)$.
\end{enumerate}
\end{theorem*}

\begin{proof}
(\textbf{Forward direction}) Let $G\in \Gthreeand$ where $G\cong G_{3'}(\tau)$. From \cref{lem:m(x)=m(y)} and \cref{lem: height3} we can assume that $\tau$ has height at most $3$ and $m_{\tau}(x)\neq m_{\tau}(y)$ for all $x,y\in V(G)$ respectively. Since $\tau$ has height at most $3$, we let $L_1\cup L_2\cup L_3$ partition $V(G)$, where $m_{\tau}(v) = \{v\}$ for all $v\in L_1$, $m_{\tau}(u)\setminus \{u\}\subseteq L_1$ for all $u\in L_2$, and $m_{\tau}(w)\setminus \{w\}\subseteq L_1\cup L_2$ for all $w\in L_3$. 

\begin{claim}
$N[v]$ induces a universe for all $v\in L_1$.
\end{claim}

\begin{proofc}
Let $v\in L_1$. Clearly $v$ is a universal vertex in $N[v]$, thus, (1) of \cref{def:universe} is satisfied. Next, let $P = \{u\in V(G) : v\in m_{\tau}(u)\}$. Since $v\in \bigcap_{u\in P}m_{\tau}(u)$, $P\subseteq N(v)$ and $P$ induces a clique. Thus, (2) of \cref{def:universe} is satisfied. Let $M = N(v)\setminus P$.

We now check condition (3) of \cref{def:universe}. To do this, it suffices to show that for any $w\in N(v)$ with $v\notin m_{\tau}(w)$ (that is, $w\in M$), there exists some $u\in P$ such that $w\in m_{\tau}(u)$. Suppose for contradiction there exists some $w\in M$ such that $w\notin \bigcup_{u\in P}m_{\tau}(u)$. Then $u\notin \ell_{\tau}(w)$ for all $u\in P$. But since $w\in M$ and since $m_{\tau}(v) = \{v\}$, we see that $m_{\tau}(\ell(w))\cap m_{\tau}(v) =  \emptyset$, contradicting that $vw\in E(G)$ and hence $\Tthreeand$-adjacent.

Finally, we consider condition (4) of \cref{def:universe}. Let $x,y\in M$ and first suppose $xy\in E(G)$. By condition (3), there exists $u_x,u_y\in P$ such that $x\in m_{\tau}(u_x)$ and $y\in m_{\tau}(u_y)$. By \cref{obs: topology}(\ref{obs:intersection}), $m_{\tau}(x)\subset m_{\tau}(u_x)$ and $m_{\tau}(y)\subset m_{\tau}(u_y)$ which implies that $u_x\in \ell_{\tau}(x)$ and $u_y\in \ell_{\tau}(y)$. Moreover there must be some $u_x = u_y$ as desired, since otherwise $m_{\tau}(\ell(x))\cap m_{\tau}(y) = \emptyset$ contradicting that $xy\in E(G)$. Likewise, if $xy\notin E(G)$, then all such $u_x$ and $u_y$ must be distinct, since if there is a common $u_x = u_y$, $m_{\tau}(\ell(x))\cap m_{\tau}(y)\neq \emptyset$ and $m_{\tau}(\ell(y))\cap m_{\tau}(x) \neq \emptyset$ contradicting that $xy\notin E(G)$. 

Thus, $(v,P,M)$ induces a universe as desired.
\end{proofc}

Let $L_1 = \{v_1,...,v_t\}$ and by the previous claim, for each $v_i$, let $U_i = (v_i,P_i,M_i)$ be the universe $N[v_i]$ induces. Let $S = \{v_1,...,v_t\}$ and let $P = \bigcup_{i=1}^t P_i$. What remains to be shown is that $\mathcal{U} = U_1,...,U_t$ cover the edges of $G$ and moreover the hypothesis conditions (1) and (2) are satisfied. Since for any $v\in L_1$, $m_{\tau}(v) = \{v\}$, we see that $v\notin P_i$ for any $i\in [t]$. Also, for any $u\in L_2\cup L_3$, there exists some $v_i\in L_1\cap N(u)$ by definition of being in $L_2\cup L_3$. Hence, $u\in P_i$ and thus, $S\cup P$ partitions $V(G)$, satisfying (1). Next, suppose $v_i,v_j\in L_1$ with $w\in N(v_i)\cap N(v_j)\cap P$. Since $w\in N(v_i)\cap N(v_j)\cap P$, we must have $v_i,v_j\in m_{\tau}(w)$. But then $m_{\tau}(\ell(v_i))\cap m_{\tau}(v_j)\neq \emptyset$ and vice versa, implying that $v_iv_j\in E(G)$. Thus, (2) is satisfied.

Finally, consider some edge $uv\in E(G)$. If $|\{u,v\}\cap L_1|>0$, say without loss of generality $v\in L_1$, then trivially we cover the edge $uv$ in $N[v]$, so we may now suppose $u,v\in L_2\cup L_3$. If $u\in L_3$ then since $m_{\tau}(u)=m_{\tau}(\ell(u))$, there exists some $v_i \in m_{\tau}(u)\cap m_{\tau}(v) \cap L_1$. Thus, $uv$ is covered by $N[v_i]$. Finally, suppose $u,v\in L_2$. Since $uv\in E(G)$, either there exists $v_i\in m_{\tau}(u)\cap m_{\tau}(v)\cap L_1$ or there exists $w\in L_3$ such that $u,v\in m_{\tau}(w)$. In the former case, $uv$ is covered by $N[v_i]$ and in the latter case, by definition of $w\in L_3$, there exists some $v_i\in m_{\tau}(w)$ and $uv$ is covered by $N[v_i]$. This shows $\mathcal{U}$ covers $E(G)$ and completes the forwards direction.

(\textbf{Backward direction}) Suppose $G$ is a graph with edge-covering of universes $U_1,...,U_t$ satisfying conditions (1) and (2). We define the following minimal base for a topology $\tau$ on $V(G)$. For notational convenience, for $u\in P$, let $P(u) = \{i : u\in P_i\}$. We now let $L_1 = S$, let $L_2 = \{u\in P: \forall i\in P(u), N_{U_i}(u)\cap M_i\subseteq S\}$ and let $L_3 = V(G)\setminus (L_1\cup L_2)$. We define the minimal sets of $\tau$ as follows.

\[m_{\tau}(v) = \begin{cases}
    \{v\} \hspace{3cm} &\text{$v\in L_1$}\\
    \{v\}\cup \bigcup_{i\in P(u)} s_i &\text{$v\in L_2$}\\
    \{v\}\cup \bigcup_{i\in P(u)} s_i\cup M_i &\text{$v\in L_3$}\\
\end{cases}\]

We now verify that $G\cong G_{3'}(\tau)$. We first show $uv\in E(G) \implies uv\in E(G_{3'}(\tau))$ and second if $uv\notin E(G) \implies uv\notin E(G_{3'}(\tau))$.

\begin{claim}
If $uv\in E(G)$ then $uv\in E(G_{3'}(\tau))$.
\end{claim}

\begin{proofc}
Since $\mathcal{U}$ covers the edges of $G$, we may assume $uv$ is covered by $U_i$ for some $i\in [t]$. We now consider various cases for which sets $u$ and $v$ may be in. First suppose, $u$ or $v$ is $s_i$. Suppose without loss of generality $u = s_i$. If $v\in P_i$, then, $v\in L_2\cup L_3$ and by the definition of $\tau$, $u\in m_{\tau}(v)\implies uv\in E(G_{3'}(\tau))$. If $v\in M_i$, then by the definition of a universe, there must exist some $w\in P_i$ with $w\sim v$. If $v\notin S$, then $w\in L_3$ and by definition of $\tau$, $u,v\in m_{\tau}(w) \implies uv\in E(G_{3'}(\tau))$ by \cref{obs: topology}(\ref{obs:T3-and}). If $v\in S$, then again we get $u,v\in m_{\tau}(w) \implies uv\in E(G_{3'}(\tau))$ by \cref{obs: topology}(\ref{obs:T3-and}). 

Now suppose $u\in P_i$. If $v\in P_i$ then $s_i\in m_{\tau}(u)\cap m_{\tau}(v)$ and thus $uv\in E(G_{3'}(\tau))$. If $v\in M_i\cap S$, then since $u\in L_2\cup L_3$, $v\in m_{\tau}(u)\implies uv\in E(G_{3'}(\tau))$. If $v\in M_i\cap P$, then $u\in L_3$, and we still see that $v\in m_{\tau}(u)\implies uv\in E(G_{3'}(\tau))$. 

Finally, suppose $u,v\in M_i$. By the definition of a universe, there exists $w\in P_i$ with $w\sim u$ and $w\sim v$. If $u,v\in S$, then since $w\in L_2\cup L_3$, $u,v\in m_{\tau}(w)$ and by \cref{obs: topology}(\ref{obs:T3-and}), $uv\in E(G_{3'}(\tau))$. If either $u$ or $v$ is in $P$, then $w\in L_3$, and again we see that $u,v\in m_{\tau}(w)$ and  by \cref{obs: topology}(\ref{obs:T3-and}), $uv\in E(G_{3'}(\tau))$. This covers all possible types of edges in $G$ and thus completes the claim.
\end{proofc}

\begin{claim}
If $uv\notin E(G)$ then $uv\notin E(G_{3'}(\tau))$.
\end{claim}

\begin{proofc}
First suppose $u,v\in S$. We assume for contradiction that $uv\in E(G_{3'}(\tau))$ Since $u,v\in S,m_{\tau}(u) = \{u\}$ and $m_{\tau}(v)=\{v\}$. Thus, if $u$ and $v$ are $\Tthreeand$-adjacent, there must exist some $w\in V(G)$ such that $u,v\in m_{\tau}(w)$. Since $u,v\in m_{\tau}(w)$, $w\in L_2\cup L_3$ which implies that $w\in P$. But this means that $w\in N(u)\cap N(v)\cap P$ contradicting condition (2).

Next, suppose $u\in S$ and $v\in P$. Let $u = s_i$. If $v\in L_3$, then $m_{\tau}(\ell(v)) = m_{\tau}(v)$, so if we suppose for contradiction that $uv\in E(G_{3'}(\tau))$, it must be the case that $u\in m_{\tau}(v)$. But this implies that $uv\in E(G)$, a contradiction. If $v\in L_2$ and $u\notin m_{\tau}(v)$, then again supposing for contradiction that $uv\in E(G_{3'}(\tau))$, there must exists some $w\in L_3$ with $u,v\in m_{\tau}(w)$. But this means that $w\in N(u)\cap N(v)\cap P$ which implies, since $v\in L_2$, there exists some $v'\in L_1 = S$ also with $w\in N(u)\cap N(v')\cap P$, contradicting condition (2).

Finally, suppose $u,v\in P$. Then, $m_{\tau}(u)\cap m_{\tau}(v) = \emptyset$ since otherwise there exists some $s_i$ such that $u,v\in P_i$ implying $uv\in E(G)$, a contradiction. If either $u$ or $v$ is in $L_3$, say $u$,  then $m_{\tau}(\ell(u)) = m_{\tau}(u)$ and as stated previously, $m_{\tau}(\ell(u))\cap m_{\tau}(v) = \emptyset \implies uv\notin E(G_{3'}(\tau))$. Hence, we may assume $u,v\in L_2$. Suppose for contradiction that $uv\in E(G_{3'}(\tau))$. Then there must exist some $w\in L_3$ with $u,v\in m_{\tau}(w)$. This contradicts (2) since $w\in N(u)\cap N(v) \cap P$ which implies, since $u,v\in L_2$, there exists some $u',v'\in L_1 = S$ with $w\in N(u')\cap N(v')\cap P$. Thus, $uv\notin E(G_{3'}(\tau))$. This covers all possible cases and thus completes the claim.
\end{proofc}
By the two previous claims, $G\cong G_{3'}(\tau)$ and thus $G\in \Gthreeand$ as desired.
\end{proof}

\section{Acknowledgments}
This problem originated during the Auburn University Student Research Seminar. We would like to specifically thank Joe Briggs and Chris Wells for organizing the seminar and for their useful conversations during the early stages of this project.

\bibliographystyle{abbrv}
\bibliography{bib}

\end{document}

%% file: figures/separation-axioms.tex
\begin{tikzpicture}


\filldraw[
  pattern={Lines[angle=45, line width=1pt, distance = 2pt]},
  pattern color=blue,
  opacity=.5,
  draw=black,
  rounded corners=10pt,
  thick
]  (0,0) -- (1,0) -- (1,1) -- (0,1) -- cycle;

\filldraw[opacity = .7, fill=blue!40, draw=black, rounded corners=10pt] (0,0) -- (1,0) -- (1,1) -- (0,1) -- cycle;

\filldraw[opacity = .7, fill=red!60, draw=black, rounded corners=6pt] (.25,.125) -- (.8,.125) -- (.8,.7) -- (.25,.7) -- cycle;

\filldraw[
  pattern={Lines[angle=-45, line width=1pt, distance = 2pt]},
  pattern color=red,
  opacity=.5,
  draw=black,
  rounded corners=6pt,
  thick
] (.25,.125) -- (.8,.125) -- (.8,.7) -- (.25,.7) -- cycle;

\draw[fill = black] (.25,.8) circle (1.5pt);
\draw[fill = black] (.5,.35) circle (1.5pt);

\node at (.4,.85) {\tiny{$x$}};
\node at (.6,.5) {\tiny{$y$}};

\node at (.5,-.6) {\tiny{$T_0$}};

\node at (.5,-.25) {\small{$\underbrace{\hspace{.75cm}}$}};



\filldraw[opacity = .5, fill=blue!60, draw=black, rounded corners=10pt] (1.5,0) -- (2.5,0) -- (2.5,1) -- (1.5,1) -- cycle;

  \filldraw[pattern={Lines[angle=45, line width=1pt, distance = 2pt]},
  pattern color=blue,
  opacity=.2,
  draw=black,
  rounded corners=10pt,
  thick] (1.5,0) -- (2.5,0) -- (2.5,1) -- (1.5,1) -- cycle;

\filldraw[opacity = .7, fill=red!70, draw=black, rounded corners=6pt] (2,.2) -- (3,.2) -- (3,.8) -- (2,.8) -- cycle;

\filldraw[pattern={Lines[angle=-45, line width=1pt, distance = 2pt]},
  pattern color=red,
  opacity=.5,
  draw=black,
  rounded corners=6pt,
  thick] (2,.2) -- (3,.2) -- (3,.8) -- (2,.8) -- cycle;

\draw[opacity = 1, draw=black, rounded corners=10pt] (1.5,0) -- (2.5,0) -- (2.5,1) -- (1.5,1) -- cycle;

\draw[fill = black] (1.75,.45) circle (1.5pt);
\draw[fill = black] (2.75,.45) circle (1.5pt);

\node at (1.75,.65) {\tiny{$x$}};
\node at (2.75,.65) {\tiny{$y$}};


\node at (2.25,-.6) {\tiny{$T_1$}};

\node at (2.25,-.25) {\small{$\underbrace{\hspace{1.25cm}}$}};


\filldraw[opacity = .7, fill=red!70, draw=black, rounded corners=7pt] (4.25,.2) -- (5,.2) -- (5,.8) -- (4.25,.8) -- cycle;

\filldraw[pattern={Lines[angle=-45, line width=1pt, distance = 2pt]},
  pattern color=red,
  opacity=.3,
  draw=black,
  rounded corners=7pt,
  thick] (4.25,.2) -- (5,.2) -- (5,.8) -- (4.25,.8) -- cycle;

\draw[opacity = 1, draw=black, rounded corners=7pt] (4.25,.2) -- (5,.2) -- (5,.8) -- (4.25,.8) -- cycle;

\filldraw[opacity = .5, fill=blue!60, draw=black, rounded corners=8pt] (3.375,0) -- (4.125,0) -- (4.125,1) -- (3.375,1) -- cycle;

\filldraw[pattern={Lines[angle=45, line width=1pt, distance = 2pt]},
  pattern color=blue,
  opacity=.2,
  draw=black,
  rounded corners=8pt,
  thick] (3.375,0) -- (4.125,0) -- (4.125,1) -- (3.375,1) -- cycle;

\draw[opacity = 1, draw=black, rounded corners=8pt] (3.375,0) -- (4.125,0) -- (4.125,1) -- (3.375,1) -- cycle;

\draw[fill = black] (3.75,.45) circle (1.5pt);
\draw[fill = black] (4.625,.45) circle (1.5pt);

\node at (3.75,.65) {\tiny{$x$}};
\node at (4.625,.65) {\tiny{$y$}};

\node at (4.2,-.6) {\tiny{$T_2$}};

\node at (4.2,-.25){\small{$\underbrace{\hspace{1.5cm}}$}};





\filldraw[opacity = .5, fill=blue!60, draw=black, rounded corners=7pt] (6.5-1.2,.2) -- (7-1.2,.2) -- (7-1.2,.8) -- (6.5-1.2,.8) -- cycle;

\filldraw[pattern={Lines[angle=45, line width=1pt, distance = 2pt]},
  pattern color=blue,
  opacity=.2,
  draw=black,
  rounded corners=7pt,
  thick] (6.5-1.2,.2) -- (7-1.2,.2) -- (7-1.2,.8) -- (6.5-1.2,.8) -- cycle;

\draw[opacity = 1, draw=black, rounded corners=7pt] (6.5-1.2,.2) -- (7-1.2,.2) -- (7-1.2,.8) -- (6.5-1.2,.8) -- cycle;

\filldraw[opacity = .7, fill=red!70, draw=black, rounded corners=8pt] (5.375+.65,0) -- (6.375+.65,0) -- (6.375+.65,1) -- (5.375+.65,1) -- cycle;

\filldraw[pattern={Lines[angle=-45, line width=1pt, distance = 2pt]},
  pattern color=red,
  opacity=.3,
  draw=black,
  rounded corners=8pt,
  thick] (5.375+.65,0) -- (6.375+.65,0) -- (6.375+.65,1) -- (5.375+.65,1) -- cycle;

\draw[opacity = 1, draw=black, rounded corners=8pt] (5.375+.65,0) -- (6.375+.65,0) -- (6.375+.65,1) -- (5.375+.65,1) -- cycle;

\filldraw[thick, opacity = 1, fill=gray!50, draw=black, rounded corners=0pt] (5.625+.65,.25) -- (6.125+.65,.25) -- (6.125+.65,.75) -- (5.625+.65,.75) -- cycle;

\draw[fill = black] (6.75-1.2,.4) circle (1.5pt);

\draw[fill = black] (6.525,.4) circle (1.5pt);


\node at (6.525,.6) {\tiny{$y$}};
\node at (5.55,.6) {\tiny{$x$}};

\node at (6.125,-.6) {\tiny{$T_3$}};

\node at (6.125,-.25) {\small{$\underbrace{\hspace{1.5cm}}$}};


\filldraw[opacity = .7, fill=red!70, draw=black, rounded corners=7pt] (8.5,0) -- (9.5,0) -- (9.5,1) -- (8.5,1) -- cycle;

\filldraw[pattern={Lines[angle=-45, line width=1pt, distance = 2pt]},
  pattern color=red,
  opacity=.3,
  draw=black,
  rounded corners=7pt,
  thick] (8.5,0) -- (9.5,0) -- (9.5,1) -- (8.5,1) -- cycle;

\draw[opacity = 1, draw=black, rounded corners=7pt] (8.5,0) -- (9.5,0) -- (9.5,1) -- (8.5,1) -- cycle;

\filldraw[opacity = .5, fill=blue!70, draw=black, rounded corners=8pt] (7.375,0) -- (8.375,0) -- (8.375,1) -- (7.375,1) -- cycle;

\filldraw[pattern={Lines[angle=45, line width=1pt, distance = 2pt]},
  pattern color=blue,
  opacity=.2,
  draw=black,
  rounded corners=8pt,
  thick] (7.375,0) -- (8.375,0) -- (8.375,1) -- (7.375,1) -- cycle;

\draw[opacity = 1, draw=black, rounded corners=8pt] (7.375,0) -- (8.375,0) -- (8.375,1) -- (7.375,1) -- cycle;

\filldraw[thick,opacity = 1, fill=gray!50, draw=black, rounded corners=0pt] (7.625,.25) -- (8.125,.25) -- (8.125,.75) -- (7.625,.75) -- cycle;

\filldraw[thick, opacity = 1, fill=gray!50, draw=black, rounded corners=0pt] (8.75,.25) -- (9.25,.25) -- (9.25,.75) -- (8.75,.75) -- cycle;


\draw[fill = black] (7.875,.4) circle (1.5pt);
\draw[fill = black] (9,.4) circle (1.5pt);

\node at (7.875,.6) {\tiny{$x$}};
\node at (9,.6) {\tiny{$y$}};

\node at (8.4,-.6) {\tiny{$T_4$}};

\node at (8.4,-.25) {\small{$\underbrace{\hspace{1.85cm}}$}};


\end{tikzpicture}

%% file: figures/Ti-examples.tex
\begin{tikzpicture}

\begin{scope}[shift={(0, 0)}] 

\draw[fill = black] (0,1) circle (3pt);
\draw[fill = black] (-.95,.31) circle (3pt);
\draw[fill = black] (-.59,-.81) circle (3pt);
\draw[fill = black] (.59,-.81) circle (3pt);
\draw[fill = black] (.95,.31) circle (3pt);

\node at (0,1.35) {$1$};
\node at (1.25,.6) {$2$};
\node at (1,-1) {$3$};
\node at (-1,-1) {$4$};
\node at (-1.25,.6) {$5$};

\node at (0,-1.75) {$G_0(\tau)$};

\end{scope}

\begin{scope}[shift={(4, 0)}] 

\draw[] (0,1) to (.95,.31);
\draw[] (.59,-.81) to (.95,.31);
\draw[] (.59,-.81) to (-.59,-.81);
\draw[] (-.95,.31) to (-.59,-.81);

\draw[fill = black] (0,1) circle (3pt);
\draw[fill = black] (-.95,.31) circle (3pt);
\draw[fill = black] (-.59,-.81) circle (3pt);
\draw[fill = black] (.59,-.81) circle (3pt);
\draw[fill = black] (.95,.31) circle (3pt);

\node at (0,-1.75) {$G_1(\tau)$};

\node at (0,1.35) {$1$};
\node at (1.25,.6) {$2$};
\node at (1,-1) {$3$};
\node at (-1,-1) {$4$};
\node at (-1.25,.6) {$5$};

\end{scope}

\begin{scope}[shift={(8, 0)}] 

\draw[] (0,1) to (.95,.31);
\draw[] (.59,-.81) to (.95,.31);
\draw[] (.59,-.81) to (-.59,-.81);
\draw[] (-.95,.31) to (-.59,-.81);

\draw[] (.95,.31) to (-.59,-.81);

\draw[fill = black] (0,1) circle (3pt);
\draw[fill = black] (-.95,.31) circle (3pt);
\draw[fill = black] (-.59,-.81) circle (3pt);
\draw[fill = black] (.59,-.81) circle (3pt);
\draw[fill = black] (.95,.31) circle (3pt);

\node at (0,-1.75) {$G_2(\tau)$};

\node at (0,1.35) {$1$};
\node at (1.25,.6) {$2$};
\node at (1,-1) {$3$};
\node at (-1,-1) {$4$};
\node at (-1.25,.6) {$5$};

\end{scope}

\begin{scope}[shift={(4, -4)}] 

\draw[] (0,1) to (.95,.31);
\draw[] (.59,-.81) to (.95,.31);
\draw[] (.59,-.81) to (-.59,-.81);
\draw[] (-.95,.31) to (-.59,-.81);

\draw[] (.95,.31) to (-.59,-.81);

\draw[] (0,1) to (.59,-.81);
\draw[] (0,1) to (-.59,-.81);
\draw[] (-.95,.31) to (.59,-.81);
\draw[] (-.95,.31) to (.95,.31);

\draw[fill = black] (0,1) circle (3pt);
\draw[fill = black] (-.95,.31) circle (3pt);
\draw[fill = black] (-.59,-.81) circle (3pt);
\draw[fill = black] (.59,-.81) circle (3pt);
\draw[fill = black] (.95,.31) circle (3pt);

\node at (0,-1.75) {$G_{3''}(\tau)$};

\node at (0,1.35) {$1$};
\node at (1.25,.6) {$2$};
\node at (1,-1) {$3$};
\node at (-1,-1) {$4$};
\node at (-1.25,.6) {$5$};

\end{scope}

\begin{scope}[shift={(0, -4)}] 

\draw[] (0,1) to (.95,.31);
\draw[] (.59,-.81) to (.95,.31);
\draw[] (.59,-.81) to (-.59,-.81);
\draw[] (-.95,.31) to (-.59,-.81);

\draw[] (.95,.31) to (-.59,-.81);

\draw[] (0,1) to (.59,-.81);
\draw[] (-.95,.31) to (.59,-.81);

\draw[fill = black] (0,1) circle (3pt);
\draw[fill = black] (-.95,.31) circle (3pt);
\draw[fill = black] (-.59,-.81) circle (3pt);
\draw[fill = black] (.59,-.81) circle (3pt);
\draw[fill = black] (.95,.31) circle (3pt);

\node at (0,-1.75) {$G_{3'}(\tau)$};

\node at (0,1.35) {$1$};
\node at (1.25,.6) {$2$};
\node at (1,-1) {$3$};
\node at (-1,-1) {$4$};
\node at (-1.25,.6) {$5$};

\end{scope}

\begin{scope}[shift={(8, -4)}] 

\draw[] (0,1) to (.95,.31);
\draw[] (.59,-.81) to (.95,.31);
\draw[] (.59,-.81) to (-.59,-.81);
\draw[] (-.95,.31) to (-.59,-.81);

\draw[] (.95,.31) to (-.59,-.81);

\draw[] (0,1) to (.59,-.81);
\draw[] (0,1) to (-.59,-.81);
\draw[] (-.95,.31) to (.59,-.81);
\draw[] (-.95,.31) to (.95,.31);

\draw[] (-.95,.31) to (0,1);

\draw[fill = black] (0,1) circle (3pt);
\draw[fill = black] (-.95,.31) circle (3pt);
\draw[fill = black] (-.59,-.81) circle (3pt);
\draw[fill = black] (.59,-.81) circle (3pt);
\draw[fill = black] (.95,.31) circle (3pt);

\node at (0,1.35) {$1$};
\node at (1.25,.6) {$2$};
\node at (1,-1) {$3$};
\node at (-1,-1) {$4$};
\node at (-1.25,.6) {$5$};

\node at (0,-1.75) {$G_4(\tau)$};
\end{scope}

\end{tikzpicture}

%% file: figures/adjusting-top-1.tex
\begin{tikzpicture}

\begin{scope}[shift={(0,0)}]
\draw (0,0) ellipse (3.5cm and 2cm);
\draw[rotate around={-18:(0,0)}] (-1.5,-1.3) ellipse (1.2cm and .6cm);

\draw[rotate around={-2:(0,0)}] (-.5,-1.1) ellipse (1cm and .5cm);

\draw[rotate around={22:(0,0)}] (1.5,-1.25) ellipse (.7cm and .3cm);

\draw[rotate around={-0:(0,0)}] (-1.05,-1) ellipse (.25 cm and .25cm);

\draw[fill = black] (-2,1-.11) circle (3pt);
\draw[fill = black] (-1,1.25-.11) circle (3pt);
\draw[fill = black] (1,1.25-.11) circle (3pt);
\draw[fill = black] (2,1-.11) circle (3pt);

\draw[fill = black] (-2.5,-.55) circle (3pt);
\draw[fill = black] (-2,-.75) circle (3pt);
\draw[fill = black] (-1.05,-1) circle (3pt);
\draw[fill = black] (-.2,-1.1) circle (3pt);
\draw[fill = black] (1.9,-.6) circle (3pt);

\node at (1,1.65-.11) {\large{$x$}};
\node at (2,1.35-.11) {\large{$y$}};
\node at (0,.5) {$\underbrace{\hspace{4cm}}$};
\node at (0,1.25) {\large{$\cdots$}};
\node at (0,0) {\large{$S_x$}};

\node at (0,-2.5) {\large{$\tau$}};

\end{scope}


\begin{scope}[shift={(7.5,0)}]
\draw (0,0) ellipse (3.5cm and 2cm);

\draw[very thick, dashed,red, domain=65:390, samples = 200] plot ({3.4*cos(\x)+0}, {1.9*sin(\x)+0});

\draw[very thick, red, dashed] (1.424,1.738) .. controls (1.6,0) .. (2.97,.95);

\draw[rotate around={-18:(0,0)}] (-1.5,-1.3) ellipse (1.2cm and .6cm);

\draw[rotate around={-2:(0,0)}] (-.5,-1.1) ellipse (1cm and .5cm);

\draw[rotate around={22:(0,0)}] (1.5,-1.25) ellipse (.7cm and .3cm);

\draw[rotate around={-0:(0,0)}] (-1.05,-1) ellipse (.25 cm and .25cm);

\draw[fill = black] (-2,1-.11) circle (3pt);
\draw[fill = black] (-1,1.25-.11) circle (3pt);
\draw[fill = black] (1,1.25-.11) circle (3pt);
\draw[fill = black] (2,1-.11) circle (3pt);

\draw[fill = black] (-2.5,-.55) circle (3pt);
\draw[fill = black] (-2,-.75) circle (3pt);
\draw[fill = black] (-1.05,-1) circle (3pt);
\draw[fill = black] (-.2,-1.1) circle (3pt);
\draw[fill = black] (1.9,-.6) circle (3pt);

\node at (1,1.65-.11) {\large{$x$}};
\node at (2,1.35-.11) {\large{$y$}};
\node at (0,1.25-.11) {\large{$\cdots$}};

\node at (0,-2.5) {\large{$\tau'$}};

\end{scope}

\end{tikzpicture}

%% file: figures/height-red.tex
\begin{tikzpicture}

\begin{scope}[shift={(0,0)}]
\draw (0,0) ellipse (1.5cm and 1cm);
\draw (-.4,0) ellipse (1cm and .75cm);
\draw (-.8,0) ellipse (.5cm and .5cm);

\draw[fill = black] (-.8,-.1) circle (3pt);
\draw[fill = black] (.1,-.1) circle (3pt);
\draw[fill = black] (1,-.1) circle (3pt);

\node at (-.8,.25) {$v_1$};
\node at (.1,.25) {$v_2$};
\node at (1,.25) {$v_3$};
\end{scope}


\begin{scope}[shift={(4,0)}]

\draw[rotate around={30:(0,0)}] (.2,.4) ellipse (1.5cm and .75cm);

\draw[rotate around={-30:(0,0)}] (.2,-.4) ellipse (1.5cm and .75cm);

\draw (-.8,0) ellipse (.5cm and .5cm);

\draw[fill = black] (-.8,-.1) circle (3pt);
\draw[fill = black] (.5,.65) circle (3pt);
\draw[fill = black] (.5,-.85) circle (3pt);

\node at (-.8,.25) {$v_1$};
\node at (.5, 1) {$v_2$};
\node at (.5,-.5) {$v_3$};
\end{scope}

\node at (2,0) {$\rightarrow$};
\node at (11.85,0) {$\rightarrow$};


\begin{scope}[shift={(9,0)}]

\draw (0.4,0) ellipse (2cm and 1.25cm);
\draw (0,0) ellipse (1.5cm and 1cm);
\draw (-.4,0) ellipse (1cm and .75cm);
\draw (-.8,0) ellipse (.5cm and .5cm);

\draw[fill = black] (-.8,-.1) circle (3pt);
\draw[fill = black] (.1,-.1) circle (3pt);
\draw[fill = black] (1,-.1) circle (3pt);
\draw[fill = black] (1.9,-.1) circle (3pt);

\node at (-.8,.25) {$v_1$};
\node at (.1,.25) {$v_2$};
\node at (1,.25) {$v_3$};
\node at (1.9,.25) {$v_4$};
\end{scope}

\begin{scope}[shift={(14,0)}]

\draw[rotate around={20:(0,0)}] (.15,.25) ellipse (1.4cm and .7cm);

\draw[rotate around={-20:(0,0)}] (.15,-.25) ellipse (1.4cm and .7cm);

\draw (0.4,0) ellipse (2cm and 1.25cm);

\draw (-.7,0) ellipse (.5cm and .5cm);

\draw[fill = black] (-.7,-.1) circle (3pt);
\draw[fill = black] (.75,.4) circle (3pt);
\draw[fill = black] (.75,-.4) circle (3pt);
\draw[fill = black] (1.75,-.1) circle (3pt);

\node at (-.7,.25) {$v_1$};
\node at (.75, .75) {$v_2$};
\node at (.75,-.75) {$v_3$};
\node at (1.75,.25) {$v_4$};

\end{scope}

\end{tikzpicture}

%% file: figures/comp-paths.tex
\begin{tikzpicture}

\node at (-.75,0) {\tiny{(i.)}};
\node at (-.75,-.75) {\tiny{(ii.)}};
\node at (-.75,-1.5) {\tiny{(iii.)}};
\node at (-.75,-2.25) {\tiny{(iv.)}};
\node at (-.75,-3) {\tiny{(v.)}};


\begin{scope}[thick, decoration={
    markings,
    mark=at position 0.55 with {\arrow[scale=.7]{<}}}
    ] 

\draw[postaction={decorate}] (0,0) -- (.5,0);   
\draw[postaction={decorate}] (1,0) -- (.5,0);   
\end{scope}

\draw[fill = black] (0,0) circle (1.5pt);
\draw[fill = black] (.5,0) circle (1.5pt);
\draw[fill = black] (1,0) circle (1.5pt);


\begin{scope}[thick, decoration={
    markings,
    mark=at position 0.55 with {\arrow[scale=.7]{>}}}
    ] 

\draw[postaction={decorate}] (0,-.75) -- (.5,-.75);   
\draw[postaction={decorate}] (1,-.75) -- (.5,-.75);   
\end{scope}

\draw[fill = black] (0,-.75) circle (1.5pt);
\draw[fill = black] (.5,-.75) circle (1.5pt);
\draw[fill = black] (1,-.75) circle (1.5pt);


\begin{scope}[thick, decoration={
    markings,
    mark=at position 0.55 with {\arrow[scale=.7]{>}}}
    ] 

\draw[postaction={decorate}] (0,-1.5) -- (.5,-1.5);   
\draw[postaction={decorate}] (1,-1.5) -- (.5,-1.5);  
\draw[postaction={decorate}] (1,-1.5) -- (1.5,-1.5);  
\end{scope}

\draw[fill = black] (0,-1.5) circle (1.5pt);
\draw[fill = black] (.5,-1.5) circle (1.5pt);
\draw[fill = black] (1,-1.5) circle (1.5pt);
\draw[fill = black] (1.5,-1.5) circle (1.5pt);


\begin{scope}[thick, decoration={
    markings,
    mark=at position 0.55 with {\arrow[scale=.7]{<}}}
    ] 

\draw[postaction={decorate}] (0,-2.25) -- (.5,-2.25);   
\draw[postaction={decorate}] (1,-2.25) -- (.5,-2.25);  
\draw[postaction={decorate}] (1,-2.25) -- (1.5,-2.25);
\draw[postaction={decorate}] (2,-2.25) -- (1.5,-2.25);
\end{scope}

\draw[fill = black] (0,-2.25) circle (1.5pt);
\draw[fill = black] (.5,-2.25) circle (1.5pt);
\draw[fill = black] (1,-2.25) circle (1.5pt);
\draw[fill = black] (1.5,-2.25) circle (1.5pt);
\draw[fill = black] (2,-2.25) circle (1.5pt);


\begin{scope}[thick, decoration={
    markings,
    mark=at position 0.55 with {\arrow[scale=.7]{>}}}
    ] 

\draw[postaction={decorate}] (0,-3) -- (.5,-3);   
\draw[postaction={decorate}] (1,-3) -- (.5,-3);  
\draw[postaction={decorate}] (1,-3) -- (1.5,-3);
\draw[postaction={decorate}] (2,-3) -- (1.5,-3);
\end{scope}

\draw[fill = black] (0,-3) circle (1.5pt);
\draw[fill = black] (.5,-3) circle (1.5pt);
\draw[fill = black] (1,-3) circle (1.5pt);
\draw[fill = black] (1.5,-3) circle (1.5pt);
\draw[fill = black] (2,-3) circle (1.5pt);

\node at (0,.25) {\tiny{$u$}};
\node at (0,.25-.75) {\tiny{$u$}};
\node at (0,.25-2*.75) {\tiny{$u$}};
\node at (0,.25-3*.75) {\tiny{$u$}};
\node at (0,.25-4*.75) {\tiny{$u$}};

\node at (.5,.25) {\tiny{$x_1$}};
\node at (.5,.25-.75) {\tiny{$x_1$}};
\node at (.5,.25-2*.75) {\tiny{$x_1$}};
\node at (.5,.25-3*.75) {\tiny{$x_1$}};
\node at (.5,.25-4*.75) {\tiny{$x_1$}};

\node at (1,.25) {\tiny{$v$}};
\node at (1,.25-.75) {\tiny{$v$}};
\node at (1,.25-2*.75) {\tiny{$x_2$}};
\node at (1,.25-3*.75) {\tiny{$x_2$}};
\node at (1,.25-4*.75) {\tiny{$x_2$}};

\node at (1.5,.25-2*.75) {\tiny{$v$}};
\node at (1.5,.25-3*.75) {\tiny{$x_3$}};
\node at (1.5,.25-4*.75) {\tiny{$x_3$}};

\node at (2,.25-3*.75) {\tiny{$v$}};
\node at (2,.25-4*.75) {\tiny{$v$}};


\begin{scope}[shift={(.5, 0)}]
\filldraw[opacity = .6, fill=gray!25, draw=black, rounded corners=7pt] (2,.25) -- (3.5,.25) -- (3.5,-.25) -- (2,-.25) -- cycle;

\filldraw[opacity = .6, fill=gray!75, draw=black, rounded corners=5.5pt] (2+.05,.25-.05) -- (2.5-.05,.25-.05) -- (2.5-.05,-.25+.05) -- (2+.05,-.25+.05) -- cycle;

\filldraw[opacity = .6, fill=gray!75, draw=black, rounded corners=5.5pt] (3+.05,.25-.05) -- (3.5-.05,.25-.05) -- (3.5-.05,-.25+.05) -- (3+.05,-.25+.05) -- cycle;

\node at (2.25,0) {\tiny{$u$}};
\node at (2.75,0) {\tiny{$x_1$}};
\node at (3.25,0) {\tiny{$v$}};
\end{scope}

\begin{scope}[shift={(.5, 0)}]
\filldraw[opacity = .6, fill=gray!25, draw=black, rounded corners=7pt] (2,.25-.75) -- (3,.25-.75) -- (3,-.25-.75) -- (2,-.25-.75) -- cycle;

\filldraw[opacity = .6, fill=gray!25, draw=black, rounded corners=7pt] (2.5,.25-.75) -- (3.5,.25-.75) -- (3.5,-.25-.75) -- (2.5,-.25-.75) -- cycle;


\filldraw[opacity = .6, fill=gray!75, draw=black, rounded corners=5.5pt] (2.5+.05,.25-.75-.05) -- (3-.05,.25-.75-.05) -- (3-.05,-.25-.75+.05) -- (2.5+.05,-.25-.75+.05) -- cycle;

\node at (2.25,0-.75) {\tiny{$u$}};
\node at (2.75,0-.75) {\tiny{$x_1$}};
\node at (3.25,0-.75) {\tiny{$v$}};
\end{scope}

\begin{scope}[shift={(.5, 0)}]
\filldraw[opacity = .6, fill=gray!25, draw=black, rounded corners=7pt] (2,.25-2*.75) -- (3,.25-2*.75) -- (3,-.25-2*.75) -- (2,-.25-2*.75) -- cycle;

\filldraw[opacity = .6, fill=gray!25, draw=black, rounded corners=7pt] (2.5,.25-2*.75) -- (4,.25-2*.75) -- (4,-.25-2*.75) -- (2.5,-.25-2*.75) -- cycle;

\filldraw[opacity = .6, fill=gray!75, draw=black, rounded corners=5.5pt] (2.5+.05,.25-2*.75-.05) -- (3-.05,.25-2*.75-.05) -- (3-.05,-.25-2*.75+.05) -- (2.5+.05,-.25-2*.75+.05) -- cycle;

\filldraw[opacity = .6, fill=gray!75, draw=black, rounded corners=5.5pt] (3.5+.05,.25-2*.75-.05) -- (4-.05,.25-2*.75-.05) -- (4-.05,-.25-2*.75+.05) -- (3.5+.05,-.25-2*.75+.05) -- cycle;

\node at (2.25,0-2*.75) {\tiny{$u$}};
\node at (2.75,0-2*.75) {\tiny{$x_1$}};
\node at (3.25,0-2*.75) {\tiny{$x_2$}};
\node at (3.75,0-2*.75) {\tiny{$v$}};
\end{scope}

\begin{scope}[shift={(.5, .75)}]

\filldraw[opacity = .6, fill=gray!25, draw=black, rounded corners=7pt] (2,.25-4*.75) -- (3.5,.25-4*.75) -- (3.5,-.25-4*.75) -- (2,-.25-4*.75) -- cycle;

\filldraw[opacity = .6, fill=gray!25, draw=black, rounded corners=7pt] (3,.25-4*.75) -- (4.5,.25-4*.75) -- (4.5,-.25-4*.75) -- (3,-.25-4*.75) -- cycle;

\filldraw[opacity = .6, fill=gray!75, draw=black, rounded corners=5.5pt] (2+.05,.25-4*.75-.05) -- (2.5-.05,.25-4*.75-.05) -- (2.5-.05,-.25-4*.75+.05) -- (2+.05,-.25-4*.75+.05) -- cycle;

\filldraw[opacity = .6, fill=gray!75, draw=black, rounded corners=5.5pt] (3+.05,.25-4*.75-.05) -- (3.5-.05,.25-4*.75-.05) -- (3.5-.05,-.25-4*.75+.05) -- (3+.05,-.25-4*.75+.05) -- cycle;

\filldraw[opacity = .6, fill=gray!75, draw=black, rounded corners=5.5pt] (4+.05,.25-4*.75-.05) -- (4.5-.05,.25-4*.75-.05) -- (4.5-.05,-.25-4*.75+.05) -- (4+.05,-.25-4*.75+.05) -- cycle;

\node at (2.25,0-4*.75) {\tiny{$u$}};
\node at (2.75,0-4*.75) {\tiny{$x_1$}};
\node at (3.25,0-4*.75) {\tiny{$x_2$}};
\node at (3.75,0-4*.75) {\tiny{$x_3$}};
\node at (4.25,0-4*.75) {\tiny{$v$}};
\end{scope}

\begin{scope}[shift={(.5, 0)}]

\filldraw[opacity = .6, fill=gray!25, draw=black, rounded corners=7pt] (2,.25-4*.75) -- (3,.25-4*.75) -- (3,-.25-4*.75) -- (2,-.25-4*.75) -- cycle;

\filldraw[opacity = .6, fill=gray!25, draw=black, rounded corners=7pt] (3.5,.25-4*.75) -- (4.5,.25-4*.75) -- (4.5,-.25-4*.75) -- (3.5,-.25-4*.75) -- cycle;

\filldraw[opacity = .6, fill=gray!25, draw=black, rounded corners=7pt] (2.5,.25-4*.75) -- (4,.25-4*.75) -- (4,-.25-4*.75) -- (2.5,-.25-4*.75) -- cycle;

\filldraw[opacity = .6, fill=gray!75, draw=black, rounded corners=5.5pt] (2.5+.05,.25-4*.75-.05) -- (3-.05,.25-4*.75-.05) -- (3-.05,-.25-4*.75+.05) -- (2.5+.05,-.25-4*.75+.05) -- cycle;

\filldraw[opacity = .6, fill=gray!75, draw=black, rounded corners=5.5pt] (3.5+.05,.25-4*.75-.05) -- (4-.05,.25-4*.75-.05) -- (4-.05,-.25-4*.75+.05) -- (3.5+.05,-.25-4*.75+.05) -- cycle;

\node at (2.25,0-4*.75) {\tiny{$u$}};
\node at (2.75,0-4*.75) {\tiny{$x_1$}};
\node at (3.25,0-4*.75) {\tiny{$x_2$}};
\node at (3.75,0-4*.75) {\tiny{$x_3$}};
\node at (4.25,0-4*.75) {\tiny{$v$}};
\end{scope}

\end{tikzpicture}

%% file: figures/construction.tex
\begin{tikzpicture}




\filldraw[opacity = .65, fill=green!60, draw=black, rounded corners=2pt] (-2.5,.5) -- (-1.75,.5) -- (-2.125,1) -- cycle; 

\filldraw[opacity = .65, fill=red!70, draw=black, rounded corners=2pt] (-2.5,.5) -- (-1.75,.5) -- (-2.125,0) -- cycle; 

\filldraw[opacity = .65, fill=blue!60, draw=black, rounded corners=2pt] (-1.75,.5) -- (-1.3,.85) --  (-1.3, .15) --cycle;

\draw[thick] (-2.5,.5) to (-1.75,.5);
\draw[thick] (-2.5,.5) to (-2.125,1);
\draw[thick] (-2.5,.5) to (-2.125,0);

\draw[thick] (-1.75,.5) to (-2.125,1);
\draw[thick] (-1.75,.5) to (-2.125,0);

\draw[thick] (-1.3,.15) to (-1.3,.85);
\draw[thick] (-1.3,.15) to (-1.75,.5);
\draw[thick] (-1.3,.85) to (-1.75,.5);


\draw[fill = black] (-2.5,.5) circle (1.5pt);

\draw[fill = black] (-1.75,.5) circle (1.5pt);

\draw[fill = black] (-2.125,1) circle (1.5pt);

\draw[fill = black] (-2.125,0) circle (1.5pt);

\draw[fill = black] (-1.3,.85) circle (1.5pt);
\draw[fill = black] (-1.3,.15) circle (1.5pt);

\node at (-2.1,1.2) 
{\tiny{$v_1$}};
\node at (-2.1,-.25) {\tiny{$v_2$}};
\node at (-1.3,1.075) {\tiny{$v_3$}};





\filldraw[
  pattern={Lines[angle=45, line width=1pt, distance = 4pt]},
  pattern color=gray,
  opacity=1,
  draw=black,
  rounded corners=6pt,
  thick
] 
(1-.25,-.25)--(1+.25,-.25)--(2+.25,1-.25)--(2+.25,1+.25)--(0-.25,1+.25)--(0-.25,1-.25)--cycle;


\filldraw[fill = gray!40, opacity = .7, draw=black, rounded corners=6pt, thick] 
(0-.25,-.25)--(0+.25,-.25)--(0+.25,1-.25)--(1+.25,1-.25)-- (1+.25,1+.25)--(0-.25,1.25)--cycle;



\filldraw[pattern={dots[size=10pt, distance=2pt]},
pattern color=gray, opacity = .7, draw=black, rounded corners=6pt, thick] 
(2-.25,-.25)--(2+.25,-.25)--(2+.25,1.25)-- (2-.25,1.25)--cycle;

\draw[fill=green!50] (0,1) circle (5pt);

\draw[fill = red!60] (1,1) circle (5pt);

\draw[fill = blue!50] (2,1) circle (5pt);

\node at (0,1) {\tiny{$v_1$}};
\node at (1,1) {\tiny{$v_2$}};
\node at (2,1) {\tiny{$v_3$}};

\draw[fill = black] (0,0) circle (1.5pt);
\draw[fill = black] (1,0) circle (1.5pt);
\draw[fill = black] (2,0) circle (1.5pt);





\end{tikzpicture}

%% file: figures/the-key.tex
\begin{tikzpicture}





\node at (0,1.2) {\tiny{(1)}};
\node at (3,1.2) {\tiny{(2)}};
\node at (6,1.2) {\tiny{(3)}};

\node at (0,-3.35) {\tiny{(4)}};
\node at (3,-3.35) {\tiny{(5)}};
\node at (6,-3.35) {\tiny{(6)}};



\begin{scope}[shift={(0, -2.5)}] 

\filldraw[opacity = .2, fill=blue!100, draw=black, rounded corners=1.51pt] (-.5,.55) -- (-.5,.45)--(.5,.45) -- (.5,.55) -- cycle;

\filldraw[opacity = .3, fill=red!100, draw=black, rounded corners=1pt] (-.5,.5) -- (-.5,-.5) -- (0,0)--cycle;

\filldraw[opacity = .25, fill=green!100, draw=black, rounded corners=1pt] (.5,.5) -- (.5,-.5) -- (0,0) -- cycle;

\draw[thick] (0,0) to (.5,.5);
\draw[thick] (0,0) to (.5,-.5);
\draw[thick] (0,0) to (-.5,.5);
\draw[thick] (0,0) to (.5,.5);
\draw[thick] (0,0) to (-.5,-.5);
\draw[thick] (-.5,.5) to (.5,.5);
\draw[thick] (-.5,.5) to (-.5,-.5);
\draw[thick] (.5,.5) to (.5,-.5);

\draw[fill = black] (0,0) circle (1.51pt);

\draw[fill = black] (-.5,.5) circle (1.51pt);

\draw[fill = black] (.5,.5) circle (1.51pt);

\draw[fill = black] (-.5,-.5) circle (1.51pt);

\draw[fill = black] (.5,-.5) circle (1.51pt);
\end{scope}


\begin{scope}[shift={(3.7, 0)}] 

\draw[] (1.65,.65) to (1.65, .25);
\draw[] (2.05,.65) to (1.65, .65);
\draw[] (2.05,.65) to (1.65, .25);

\draw[fill = black] (1.65,.65) circle (1.51pt);
\draw[fill = red, red] (1.65,.25) circle (1.51pt);
\draw[blue, fill = blue] (2.05,.65) circle (1.51pt);

\draw[] (2.55,.65) to (2.95, .25);
\draw[] (2.55,.65) to (2.95, .65);
\draw[] (2.95,.65) to (2.95, .25);

\draw[fill = blue, blue] (2.55,.65) circle (1.51pt);
\draw[fill = black] (2.95,.65) circle (1.51pt);
\draw[green, fill = green] (2.95,.25) circle (1.51pt);

\draw[] (1.65,-.65) to (1.65, -.25);

\draw[fill = black] (1.65,-.65) circle (1.51pt);
\draw[fill = red, red] (1.65,-.25) circle (1.51pt);

\draw[] (2.95,-.65) to (2.95, -.25);

\draw[fill = black] (2.95,-.65) circle (1.51pt);
\draw[green, fill = green] (2.95,-.25) circle (1.51pt);


\draw[] (2.3,-.2) to (2.1,.2);
\draw[] (2.3,-.2) to (2.5,.2);
\draw[] (2.1,.2) to (2.5,.2);

\draw[fill = black] (2.3,-.2) circle (1.51pt);

\draw[fill = red, red] (2.1,.2) circle (1.51pt);

\draw[fill = green, green] (2.5,.2) circle (1.51pt);
\end{scope}


\begin{scope}[shift={(-7.5,0)}] 

    \fill[opacity = .3, fill=red!100, draw=black, rounded corners=2.5pt] (6.7,-.65) -- (6.3,.35) -- (6.25,.5) -- (6.3,.65) -- (6.5,.75) -- (7.5,.75) -- (7.7,.65) -- (7.75,.5) -- (7.7,.35) -- (7.6,.25) -- (7.1,.25) -- (6.9,-.65) -- cycle;
\draw[opacity = .9, draw=black, rounded corners=2.5pt] (6.7,-.65) -- (6.3,.35) -- (6.25,.5) -- (6.3,.65) -- (6.5,.75) -- (7.5,.75) -- (7.7,.65) -- (7.75,.5) -- (7.7,.35) -- (7.6,.25) -- (7.1,.25) -- (6.9,-.65) -- cycle;

\fill[opacity = .25, fill=green!100, draw=black, rounded corners=2.5pt] (15-6.7,-.65) -- (15-6.3,.35) -- (15-6.25,.5) -- (15-6.3,.65) -- (15-6.5,.75) -- (15-7.5,.75) -- (15-7.7,.65) -- (15-7.75,.5) -- (15-7.7,.35) -- (15-7.6,.25) -- (15-7.1,.25) -- (15-6.9,-.65) -- cycle;
\draw[opacity = .9, draw=black, rounded corners=2.5pt] (15-6.7,-.65) -- (15-6.3,.35) -- (15-6.25,.5) -- (15-6.3,.65) -- (15-6.5,.75) -- (15-7.5,.75) -- (15-7.7,.65) -- (15-7.75,.5) -- (15-7.7,.35) -- (15-7.6,.25) -- (15-7.1,.25) -- (15-6.9,-.65) -- cycle;

\fill[opacity = .2, fill=blue!100, draw=black, rounded corners=2.5pt] (7.4,-.65) --
(6.9,.25) --
 (6.8,.35) --
(6.75,.5) --
 (6.8,.65) -- 
(7,.75) --
(7.2,.65) --
 (7.25,.5) --
 (7.2,.35) -- 
(7.5,-.125) --
(15-7.2,.35) --
 (15-7.25,.5) --
(15-7.2,.65) --
 (15-7,.75) --
 (15-6.8,.65) --
(15-6.75,.5) --
 (15-6.8,.35) --
 (15-6.9,.25) --
 (15-7.4,-.65) -- cycle;

 \draw[opacity = .9, draw=black, rounded corners=2.5pt] (7.4,-.65) --
(6.9,.25) --
 (6.8,.35) --
(6.75,.5) --
 (6.8,.65) -- 
(7,.75) --
(7.2,.65) --
 (7.25,.5) --
 (7.2,.35) -- 
(7.5,-.125) --
(15-7.2,.35) --
 (15-7.25,.5) --
(15-7.2,.65) --
 (15-7,.75) --
 (15-6.8,.65) --
(15-6.75,.5) --
 (15-6.8,.35) --
 (15-6.9,.25) --
 (15-7.4,-.65) -- cycle;

\draw[fill = red, red] (6.8,-.5) circle (1.51pt);
\draw[fill = blue, blue] (7.5,-.5) circle (1.51pt);
\draw[fill = green, green] (8.2,-.5) circle (1.51pt);

\draw[fill = lightgray] (6.5,.5) circle (5pt);
\draw[fill = lightgray] (7,.5) circle (5pt);
\draw[fill = lightgray] (7.5,.5) circle (5pt);
\draw[fill = lightgray] (8,.5) circle (5pt);
\draw[fill = lightgray] (8.5,.5) circle (5pt);


\draw[fill = black] (6.5,.5) circle (1pt);
\draw[fill = black] (7,.5) circle (1pt);
\draw[fill = black] (7.5,.5) circle (1pt);

\draw[fill = black] (8,.5) circle (1pt);
\draw[fill = black] (8.5,.5) circle (1pt);

\end{scope}


\begin{scope}[shift={(-1.65, 0)}] 

\draw[] (4.15,.65) to (4.65, .45);
\draw[] (5.15,.65) to (4.65, .45);
\draw[] (4.15,.65) to (4.35, 0);
\draw[] (4.15,-.65) to (4.35, 0);
\draw[] (5.15,.65) to (4.95, 0);
\draw[] (5.15,-.65) to (4.95, 0);

\draw[] (4.65,.45) to (4.35, 0);
\draw[] (4.65,.45) to (4.95, 0);
\draw[] (4.35,0) to (4.95, 0);
\draw[] (4.35,0) to (4.65, -.25);
\draw[] (4.95,0) to (4.65, -.25);


\draw[fill = black] (4.15,.65) circle (1.51pt);

\draw[fill = black] (5.15,.65) circle (1.51pt);

\draw[fill = black] (4.15,-.65) circle (1.51pt);

\draw[fill = black] (5.15,-.65) circle (1.51pt);

\draw[fill = blue, blue] (4.65,.45) circle (1.51pt);

\draw[fill = red, red] (4.35,0) circle (1.51pt);
\draw[fill = green, green] (4.95,0) circle (1.51pt);


\draw[fill = black] (4.65,-.25) circle (1.51pt);

\end{scope}


\begin{scope}[shift={(3,.6)}] 

\draw[thick] (-.5,-2.5) to (.5,-2.5);

\draw[fill = black] (-.5,-2.5) circle (1.51pt);
\draw[fill = black] (.5,-2.5) circle (1.51pt);

\filldraw[opacity = .2, fill=blue!100, draw=black, rounded corners=1.51pt] (-.5,-2.55) -- (-.5,-2.45)--(.5,-2.45) -- (.5,-2.55) -- cycle;

\filldraw[opacity = .3, fill=red!100, draw=black, rounded corners=1.51pt] (-.5,-2.75) -- (-.5,-3.5)--(-.1,-3.125) -- cycle;

\filldraw[opacity = .2, fill=green!100, draw=black, rounded corners=1.51pt] (.5,-2.75) -- (.5,-3.5)--(.1,-3.125) -- cycle;

\draw[thick] (-.5,-2.75) to (-.5, -3.5);
\draw[thick] (-.5,-2.75) to (-.1, -3.125);
\draw[thick] (-.1,-3.125) to (-.5, -3.5);

\draw[thick] (.5,-2.75) to (.5, -3.5);
\draw[thick] (.5,-2.75) to (.1, -3.125);
\draw[thick] (.1,-3.125) to (.5, -3.5);

\draw[fill = black] (-.5,-2.75) circle (1.51pt);

\draw[fill=red, rotate around={45:(-.5 ,-3.6)}] (-.5,-3.6) rectangle +(3.51pt,3.51pt);

\draw[fill = black] (-.1,-3.125) circle (1.51pt);

\draw[fill = black] (.5,-2.75) circle (1.51pt);

\draw[fill=green, rotate around={45:(.5 ,-3.6)}] (.5,-3.6) rectangle +(3.51pt,3.51pt);

\draw[fill = black] (.1,-3.125) circle (1.51pt);

\end{scope}


\begin{scope}[shift={(3.65,.6)}] 

\draw (2.35,-2.5) to (2,-3);
\draw (2.35,-2.5) to (2.7,-3);
\draw (2,-3) to (2.7,-3);
\draw (2,-3) to (2,-3.5);
\draw (2.7,-3) to (2.7,-3.5);

\draw[fill = blue, blue] (2.35,-2.5) circle (1.51pt);
\draw[fill = red, red] (2,-3) circle (1.51pt);
\draw[fill = green, green] (2.7,-3) circle (1.51pt);
\draw[fill=red, rotate around={45:(2,-3.6)}] (2,-3.6) rectangle +(3.51pt,3.51pt);

\draw[fill=green, rotate around={45:(2.7 ,-3.6)}] (2.7,-3.6) rectangle +(3.51pt,3.51pt);
\end{scope}


\end{tikzpicture}

%% file: figures/critical-ex.tex
\begin{tikzpicture}

\begin{scope}[shift={(0, 0)}]
\draw[line width=5pt, red, opacity=0.3] (-2,0) to (0,0); 
\draw[ultra thick,red] (-2,0) to (0,0);
\draw[ultra thick] (0,0) to (2,0);
\draw[ultra thick] (0,0) to (1,1);

\draw[line width=5pt, red, opacity=0.3] (0,0) to (2,-1.5);
\draw[ultra thick, red] (0,0) to (2,-1.5);

\draw[line width=5pt, red, opacity=0.3] (0,0) to (0,-1.5);
\draw[ultra thick, red] (0,0) to (0,-1.5);

\draw[ultra thick] (2,0) to (1,1);
\draw[ultra thick] (2,0) to (2,-1.5);
\draw[ultra thick] (2,0) to (0,-1.5);

\draw[ultra thick] (2,-1.5) to (0,-1.5);

\draw[ultra thick] (1,1) to (0,-1.5);
\draw[ultra thick] (1,1) to (2,-1.5);

\draw[ultra thick] (2,0) to (4,0);
\draw[ultra thick] (2,0) to (4,-1.5);
\draw[line width=5pt, red, opacity=0.3] (2,-1.5) to (4,-1.5);
\draw[ultra thick, red] (2,-1.5) to (4,-1.5);
\draw[ultra thick] (2,-1.5) to (4,0);
\draw[ultra thick] (4,-1.5) to (4,0);

\draw[ultra thick] (6,-1) to (4,-1.5);
\draw[ultra thick] (6,-2.5) to (4,-1.5);
\draw[ultra thick] (6,-2.5) to (6,-1);

\draw[ultra thick] (3,-3.5) to (5,-3.5);
\draw[line width=5pt, red, opacity=0.3] (3,-3.5) to (4,-1.5);
\draw[ultra thick, red] (3,-3.5) to (4,-1.5);
\draw[ultra thick] (5,-3.5) to (4,-1.5);

\draw[ultra thick] (1,-3.5) to (3,-3.5);
\draw[ultra thick] (1,-3.5) to (2,-5);
\draw[line width=5pt, red, opacity=0.3] (3,-3.5) to (2,-5);
\draw[ultra thick, red] (3,-3.5) to (2,-5);
\draw[line width=5pt, red, opacity=0.3] (-1,-3) to (0,-1.5);
\draw[ultra thick, red] (-1,-3) to (0,-1.5);
\draw[ultra thick] (-1,-3) to (1,-3.5);
\draw[ultra thick] (0,-1.5) to (1,-3.5);

\draw[fill = black] (-2,0) circle (4pt);

\draw[fill = black] (0,0) circle (4pt);
\draw[fill = black] (2,0) circle (4pt);
\draw[fill = black] (1,1) circle (4pt);
\draw[fill = black] (0,-1.5) circle (4pt);
\draw[fill = black] (2,-1.5) circle (4pt);

\draw[fill = black] (4,0) circle (4pt);
\draw[fill = black] (4,-1.5) circle (4pt);

\draw[fill = black] (6,-1) circle (4pt);
\draw[fill = black] (6,-2.5) circle (4pt);

\draw[fill = black] (3,-3.5) circle (4pt);
\draw[fill = black] (5,-3.5) circle (4pt);

\draw[fill = black] (2,-5) circle (4pt);
\draw[fill = black] (1,-3.5) circle (4pt);

\draw[fill = black] (-1,-3) circle (4pt);

\node at (2,-6.5) {\huge{$G$ and $T_G$}};
\node at (-1,.35) {\Large{$C_6$}};
\node at (1,1.5) {\Large{$C_5$}};
\node at (3,.35) {\Large{$C_4$}};
\node at (5.4,-1.7) {\Large{$C_3$}};
\node at (4,-2.7) {\Large{$C_2$}};
\node at (2,-4.1) {\Large{$C_1$}};
\node at (-.05,-2.6) {\Large{$C_7$}};

\end{scope}


\begin{scope}[shift={(19, 0)}]
\draw[line width=5pt, red, opacity=0.3] (-2,0) to (0,0);
\draw[ultra thick, red] (-2,0) to (0,0);
\draw[ultra thick] (0,0) to (1,1);
\draw[line width=5pt, red, opacity=0.3] (0,0) to (2,-1.5);
\draw[ultra thick, red] (0,0) to (2,-1.5);
\draw[line width=5pt, red, opacity=0.3] (0,0) to (0,-1.5);
\draw[ultra thick, red] (0,0) to (0,-1.5);
\draw[ultra thick] (2,0) to (2,-1.5);

\draw[ultra thick, black] (2,-1.5) to (0,-1.5);
\draw[ultra thick] (1,1) to (0,-1.5);
\draw[ultra thick] (1,1) to (2,-1.5);

\draw[line width=5pt, red, opacity=0.3] (2,-1.5) to (4,-1.5);
\draw[ultra thick, red] (2,-1.5) to (4,-1.5);
\draw[ultra thick] (2,-1.5) to (4,0);
\draw[ultra thick] (4,-1.5) to (4,0);

\draw[ultra thick] (6,-1) to (4,-1.5);
\draw[ultra thick] (6,-2.5) to (4,-1.5);

\draw[ultra thick] (3,-3.5) to (5,-3.5);
\draw[line width=5pt, red, opacity=0.3] (3,-3.5) to (4,-1.5);
\draw[ultra thick, red] (3,-3.5) to (4,-1.5);
\draw[ultra thick] (5,-3.5) to (4,-1.5);

\draw[ultra thick] (1,-3.5) to (3,-3.5);
\draw[line width=5pt, red, opacity=0.3] (3,-3.5) to (2,-5);
\draw[ultra thick, red] (3,-3.5) to (2,-5);

\draw[line width=5pt, red, opacity=0.3] (-1,-3) to (0,-1.5);
\draw[ultra thick, red] (-1,-3) to (0,-1.5);

\draw[fill = black] (-2,0) circle (4pt);

\draw[fill = black] (0,0) circle (4pt);
\draw[fill = black] (2,0) circle (4pt);
\draw[fill = black, black] (1,1) circle (4pt);
\draw[fill = black] (0,-1.5) circle (4pt);
\draw[fill = black] (2,-1.5) circle (4pt);

\draw[fill = black, black] (4,0) circle (4pt);
\draw[fill = black] (4,-1.5) circle (4pt);

\draw[fill = black] (6,-1) circle (4pt);
\draw[fill = black] (6,-2.5) circle (4pt);

\draw[fill = black] (3,-3.5) circle (4pt);
\draw[fill = black, black] (5,-3.5) circle (4pt);

\draw[fill = black] (2,-5) circle (4pt);
\draw[fill = black] (1,-3.5) circle (4pt);

\draw[fill = black] (-1,-3) circle (4pt);

\node at (3.1,-4.05) {\Large{$x$}};
\node at (2.5,-5.25) {\Large{$y_1$}};
\node at (4.75,-4.05) {\Large{$w$}};

\node at (2,-6.5) {\huge{$W$}};

\end{scope}

\begin{scope}[scale = 1.4, shift={(6, 0)}]

\draw[line width=5pt, red, opacity=0.3] (-.5,0) to (1,0);
\draw[ultra thick, red] (-.5,0) to (1,0);

\draw[line width=5pt, red, opacity=0.3] (1,0) to (3,0);
\draw[ultra thick, red] (1,0) to (3,0);

\draw[line width=5pt, red, opacity=0.3] (3,0) to (5,-1);
\draw[ultra thick, red] (3,0) to (5,-1);

\draw[line width=5pt, red, opacity=0.3] (5,-1) to (3.5,-2);
\draw[ultra thick, red] (5,-1) to (3.5,-2);

\draw[ultra thick] (3,0) to (3.5,-2);

\draw[line width=5pt, red, opacity=0.3] (2,-3) to (3.5,-2);
\draw[ultra thick, red] (2,-3) to (3.5,-2);

\draw[ultra thick] (.5,-2) to (2,-3);

\draw[line width=5pt, red, opacity=0.3] (.5,-2) to (1,0);
\draw[ultra thick, red] (.5,-2) to (1,0);

\draw[fill = white] (-.5,0) circle (10pt);
\draw[fill = white] (1,0) circle (10pt);
\draw[fill = white] (3,0) circle (10pt);
\draw[fill = white] (5,-1) circle (10pt);
\draw[fill = white] (3.5,-2) circle (10pt);
\draw[fill = white] (2,-3) circle (10pt);
\draw[fill = white] (.5,-2) circle (10pt);

\node[] at (-.5,0) {\Large{$C_6$}};
\node[] at (1,0) {\Large{$C_5$}};
\node[] at (3,0) {\Large{$C_4$}};
\node[] at (5,-1) {\Large{$C_3$}};
\node[] at (3.5,-2) {\Large{$C_2$}};
\node[] at (2,-3) {\Large{$C_1$}};
\node[] at (.5,-2) {\Large{$C_7$}};

\node at (2,-4.65) {\huge{$H_G(\C)$} and $T_{\C}$};

\end{scope}

\end{tikzpicture}

%% file: figures/universe.tex
\begin{tikzpicture}

\draw[dashed] (0,0) circle (70pt);

\draw[dashed, fill = lightgray, opacity = .6] (0,0) circle (47pt);

\draw[thick] (0,1) to (-.95,.31);
\draw[thick] (0,1) to (-.59,-.81);
\draw[thick] (0,1) to (.59,-.81);
\draw[thick] (0,1) to (.95,.31);
\draw[thick] (-.95,.31) to (-.59,-.81);
\draw[thick] (-.95,.31) to (.59,-.81);
\draw[thick] (-.95,.31) to (.95,.31);
\draw[thick] (-.59,-.81) to (.59,-.81);
\draw[thick] (-.59,-.81) to (.95,.31);
\draw[thick] (.59,-.81) to (.95,.31);

\draw[thick, yellow, opacity = .9] (0,3) to (0,1);
\draw[thick, yellow, opacity = .9] (0,3) to (-.59,-.81);
\draw[thick, yellow, opacity = .9] (0,3) to (-.95,.31);
\draw[thick, yellow, opacity = .9] (0,3) to (.59,-.81);
\draw[thick, yellow, opacity = .9] (0,3) to (.95,.31);
\draw[thick, yellow, opacity = .9] (0,3) to (-1.7,1.2);
\draw[thick, yellow, opacity = .9] (0,3) to (-2.1,.2);
\draw[thick, yellow, opacity = .9] (0,3) to (-1.7,-1.2);
\draw[thick, yellow, opacity = .9] (0,3) to (1.2,1.7);
\draw[thick, yellow, opacity = .9] (0,3) to (1.9,1);

\draw[fill = black] (0,1) circle (4pt);
\draw[fill = black] (-.95,.31) circle (4pt);
\draw[fill = black] (-.59,-.81) circle (4pt);
\draw[fill = black] (.59,-.81) circle (4pt);
\draw[fill = black] (.95,.31) circle (4pt);

\node at (0,-1.25) {Planets};
\node at (0,-2.1) {Moons};

\draw[thick] (-1.7,1.2) to (-.95,.31);
\draw[thick] (-2.1,.2) to (-.95,.31);
\draw[thick] (-2.1,.2) to (-1.7,1.2);
\draw[thick] (-2.1,.2) to (-.59,-.81);
\draw[thick] (-1.7,-1.2) to (-.59,-.81);
\draw[thick] (-1.7,-1.2) to (-2.1,.2);
\draw[thick] (1.2,1.7) to (.95,.31);
\draw[thick] (1.9,1) to (.95,.31);
\draw[thick] (1.9,1) to (1.2,1.7);

\draw[fill = blue, blue] (-1.7,1.2) circle (4pt);
\draw[fill = blue, blue] (-2.1,.2) circle (4pt);
\draw[fill = blue, blue] (-1.7,-1.2) circle (4pt);
\draw[fill = blue, blue] (1.2,1.7) circle (4pt);
\draw[fill = blue, blue] (1.9,1) circle (4pt);

\draw[fill =yellow] (0,3) circle (4pt);

\node at (0,3.5) {Sun};

\end{tikzpicture}